\documentclass[journal,twoside,web]{ieeecolor}

\usepackage{generic}
\usepackage{cite}
\usepackage{amsmath,amssymb,amsfonts}
\usepackage{mathtools}
\usepackage{graphicx}
\usepackage{textcomp}

\usepackage{url}

\usepackage{amssymb}
\usepackage{booktabs}
\usepackage{multirow}

\usepackage{subcaption}
\usepackage{caption}
\graphicspath{{./figures/}}

\newcommand{\B}{{\mathbb B}} 
\newcommand{\R}{{\mathbb{R}}}
\newcommand{\Z}{{\mathbb{Z}}}

\newcommand{\nx}{n_{\mathrm{x}}}
\newcommand{\nU}{n_{\mathrm{u}}}
\newcommand{\nc}{n_{\mathrm{c}}}

\newcommand{\lxor}{\veebar}

\newcommand{\I}{\mathcal{I}}
\newcommand{\blo}{\ubar{c}}
\newcommand{\bup}{\bar{c}}

\newcommand{\Ts}{T_{\mathrm{s}}}

\newcommand{\ubar}{\underline}
\renewcommand{\bar}{\overline}

\newcommand{\nobs}{n_{\mathrm{obs}}}

\newcommand{\pX}{p_{\mathrm{X}}}
\newcommand{\pY}{p_{\mathrm{Y}}}

\newcommand{\vy}{v_{\mathrm{y}}}

\newcommand{\ps}{p_{\mathrm{s}}}
\newcommand{\pn}{p_{\mathrm{n}}}
\newcommand{\vs}{v_{\mathrm{s}}}
\newcommand{\vn}{v_{\mathrm{n}}}
\newcommand{\as}{a_{\mathrm{s}}}

\newcommand{\bps}{\bar{p}_{\mathrm{s}}}
\newcommand{\ups}{\ubar{p}_{\mathrm{s}}}
\newcommand{\bpn}{\bar{p}_{\mathrm{n}}}
\newcommand{\upn}{\ubar{p}_{\mathrm{n}}}

\newcommand{\alpham}{\alpha_{\mathrm{R}}^{\mathrm{max}}}
\newcommand{\alphar}{\alpha_{\mathrm{R}}}

\newcommand{\scol}{\nu^{\mathrm{c}}}
\newcommand{\scolx}{\nu_{\mathrm{s}}^{\mathrm{c}}}
\newcommand{\scoly}{\nu_{\mathrm{n}}^{\mathrm{c}}}

\newcommand{\pref}{p_{\mathrm{n}}^{\mathrm{ref}}}
\newcommand{\bpref}{\bar{p}_{\mathrm{n}}^{\mathrm{ref}}}
\newcommand{\bpsref}{\bar{p}_{\mathrm{s}}^{\mathrm{ref}}}

\newcommand{\tc}{t_{\mathrm{c}}}
\newcommand{\lc}{\delta_{\mathrm{c}}}

\newcommand{\deltac}{\delta_{\mathrm{c}}}
\newcommand{\lu}{\delta_{\mathrm{c}}^{\mathrm{u}}}
\newcommand{\ld}{\delta_{\mathrm{c}}^{\mathrm{d}}}
\newcommand{\wl}{w_{\mathrm{l}}}
\newcommand{\nl}{n_{\mathrm{l}}}
\newcommand{\pc}{\Delta_{\mathrm{c}}}

\newcommand{\M}{M}
\newcommand{\nz}{n_{\mathrm{z}}}

\newcommand{\flag}{\texttt{infeasible}}

\newcommand{\z}{\boldsymbol{z}}

\newcommand{\bin}{\delta}
\renewcommand{\P}{\mathcal{P}}
\newcommand{\fixed}{\mathcal{R}}

\newcommand{\yout}{\boldsymbol{y}}

\usepackage{algorithm}
\usepackage{algpseudocode}
\algrenewcommand\algorithmicrequire{\textbf{Input:}}
\algrenewcommand\algorithmicensure{\textbf{Output:}}

\algnewcommand{\IfThenElse}[3]{%
	\State \algorithmicif\ #1\ \algorithmicthen\ #2\ \algorithmicelse\ #3}
\algnewcommand{\IfThen}[2]{%
	\State \algorithmicif\ #1\ \algorithmicthen\ #2}

\newcommand{\software}[1]{\texttt{#1}}
\newcommand{\case}[1]{\raisebox{.5pt}{\textcircled{\raisebox{-.9pt} {#1}}}}

\newtheorem{Theorem}{Theorem}

\newtheorem{Remark}[Theorem]{Remark}

\newtheorem{Assumption}[Theorem]{Assumption}
\newtheorem{Definition}[Theorem]{Definition}
\newtheorem{Proposition}[Theorem]{Proposition}

\usepackage{color}
\definecolor{wheat}{rgb}{0.96,0.87,0.70}

\newcommand{\change}[1]{#1}

\def\BibTeX{{\rm B\kern-.05em{\sc i\kern-.025em b}\kern-.08em
    T\kern-.1667em\lower.7ex\hbox{E}\kern-.125emX}}
\begin{document}
\title{Real-time Mixed-Integer Quadratic Programming for Vehicle Decision Making and Motion Planning}
\author{Rien Quirynen, Sleiman Safaoui, \IEEEmembership{Member, IEEE}, and Stefano Di Cairano, \IEEEmembership{Senior Member, IEEE}
\thanks{R. Quirynen and S. Di Cairano are with 
Mitsubishi Electric Research Laboratories, Cambridge, MA, USA (e-mail: \{quirynen,dicairano\}@merl.com).}
\thanks{Sleiman Safaoui is with the Department of Electrical Engineering, University
of Texas at Dallas, Richardson, TX, USA (e-mail: sleiman.safaoui@utdallas.edu). This work was done during his internship at Mitsubishi Electric Research Laboratories.}}

\maketitle

\begin{abstract}
We develop a real-time feasible mixed-integer programming-based decision making~(MIP-DM) system for automated driving. Using a linear vehicle model in a road-aligned coordinate frame, the lane change constraints, collision avoidance and traffic rules can be formulated as mixed-integer inequalities, resulting in a mixed-integer quadratic program~(MIQP). The proposed MIP-DM simultaneously performs \change{maneuver selection and trajectory generation} by solving the MIQP at each sampling time instant. \change{While solving MIQPs in real time has been considered intractable in the past, we show that our recently developed solver \software{BB-ASIPM} is capable of solving MIP-DM problems on embedded hardware in real time.} The performance of this approach is illustrated in simulations in various scenarios including merging points and traffic intersections\change{, and hardware-in-the-loop simulations on dSPACE Scalexio and MicroAutoBox-III}. 
Finally, we present results from hardware experiments 
on small-scale automated vehicles.
\end{abstract}


\section{Introduction}
\label{sec:introduction}

\IEEEPARstart{A}{utomated} transportation systems, even in the case of partial automation, may lead to reduced road accidents and more efficient usage of the road network.
\change{However, the} complexity of automated driving~(AD) and advanced driver-assistance systems~(ADAS) and their real-time requirements in resource-limited automotive platforms~\cite{DiCairano2018tutorial} requires the implementation of a multi-layer guidance and control architecture~\cite{Paden2016,Guanetti2018}. Thus, the ADAS/AD system consists of multiple interconnected components, including communication and sensor interfaces connecting each block and potentially executing at different sampling rates, \change{aiming} for the integrated system to satisfy the driving specifications~\cite{Cairano2016,SahinQuirynenEtAl2020}. 

A typical guidance and control architecture is illustrated in Figure~\ref{fig:MIP_DM_architecture1}, e.g., similar to~\cite{Berntorp2018,Ahn2021}.
Based on a route given by a navigation system, a decision making module decides when to perform maneuvers \change{such as} lane changing, stopping, waiting, and intersection crossing. Given these decisions, a motion planning system generates a state trajectory to execute the maneuvers, and a vehicle control system computes the input signals to track the trajectory.

Optimization-based motion planning and control techniques, such as model predictive control~(MPC), \change{directly account for dynamics, constraints and objectives in a model-based design framework~\cite{Mayne2013}.}
This has been extended to hybrid systems~\cite{bemporad1999control}, including both discrete and continuous decision variables. The resulting hybrid MPC can tackle a large range of problems, including switched dynamical systems~\cite{MarcucciTedrake2019},
 motion planning with obstacle avoidance~\cite{LandryDeitsEtAl2016}, logic rules and temporal logic specifications~\cite{SahinQuirynenEtAl2020}.
However, the \change{mixed-integer optimal control problem~(MIOCP) to be solved at each step is} non-convex due to integer variables, and $\mathcal{NP}$-hard~\cite{Pia2017}. For a linear-quadratic objective, linear or piecewise-linear dynamics and inequality constraints, the \change{MIOCP results in} a mixed-integer quadratic program~(MIQP).

Recent work~\cite{Quirynen2022} indicates that, by exploiting the particular structure of \change{the MIOCPs}, real-time solvers can achieve performance comparable to commercial tools, e.g., \software{GUROBI}~\cite{gurobi} and \software{MOSEK}~\cite{mosek}, especially for small to medium-scale problems. Therefore, we use the tailored \software{BB-ASIPM} solver~\cite{Quirynen2022}, using a branch-and-bound~(B\&B) method with reliability branching and warm starting~\cite{Hespanhol2019}, block-sparse presolve techniques~\cite{Quirynen2022}, early termination and infeasibility detection~\cite{Liang2021} within a fast convex quadratic programming~(QP) solver based on an active-set interior point method~(\software{ASIPM})~\cite{ASIPM}.

In this paper, we design a mixed-integer programming decision making~(MIP-DM) module \change{for vehicles} that simultaneously computes a sequence of discrete decisions and a continuous motion trajectory \change{in a hybrid MPC framework}. This approach eliminates the need for a \change{separate} motion planner \change{in the ADAS/AD architecture} as long as an advanced vehicle control algorithm is used, e.g., based on nonlinear MPC~(NMPC), see Figure~\ref{fig:MIP_DM_architecture2}. We demonstrate the proposed MIP-DM approach in simulations in various scenarios including merging points and traffic intersections, and we confirm its real-time feasibility on dSPACE Scalexio and MicroAutoBox-III rapid prototyping units \change{commonly used in automotive development}. Finally, we present results from hardware experiments using MIP-DM in combination with NMPC-based reference tracking on a setup with small-scale automated vehicles.

\subsection{Relation with Existing Literature}
In the DARPA Urban Challenge~\cite{DARPA2009}, most teams implemented rule-based decision making systems involving hand-tuned heuristics for different urban-driving scenarios.
Some recent works on vehicle decision making are based on machine learning, e.g., supervised or reinforcement learning~\cite{schwarting2018planning,Liu2021}, \change{which lacks guarantees}. 
The work in~\cite{Ahn2021} proposes the use of automata combined with set reachability, \change{however it does not account for performance, but only for maneuver feasibility}.
The work in~\cite{Esterle2018} proposes a method for simultaneous trajectory generation and maneuver selection, but the complexity of the approach grows rapidly with the number of obstacles.

Our prior work~\cite{SahinQuirynenEtAl2020} proposed to define traffic rules as signal temporal logic~(STL) formulae that are converted into a set of mixed-integer inequalities for vehicle decision making based on the solution of MIQPs. \change{This results in formal guarantees but using an excessively large optimization problem for real-time implementation, in part due to the automated STL formulae translation.} Motivated by the latter results, the present paper proposes a real-time feasible MIQP formulation for vehicle decision making and motion planning. An overview on MIP-based decision making, motion planning and control problems may be found in~\cite{Richards2005,Ioan2021}. Specifically for ADAS/AD systems, the works in~\cite{Ballesteros2017,Miller2018} propose MIPs for vehicle lane changing and overtaking maneuvers. To the best of our knowledge, this paper presents the first MIP for decision making with an embedded solver that is demonstrated to be real-time feasible in \change{automotive} hardware-in-the-loop~(HIL) simulations and in small-scale vehicle experiments.

\subsection{Contributions of Present Work}
A first contribution of the present paper is a detailed description of an MIQP formulation for vehicle decision making that can handle a wide range of traffic scenarios, while operating in a dynamic environment with potentially changing traffic rules. Second, we present the tailored \software{BB-ASIPM} solver and illustrate its computational performance to implement the proposed MIP-DM method, comparing against state-of-the-art software tools based on simulation results in various scenarios including merging points and traffic intersections. Third, we illustrate real-time feasibility of the approach on dSPACE Scalexio and MicroAutoBox-III rapid prototyping units. A fourth contribution includes the results from hardware experiments based on MIP-DM in combination with NMPC-based reference tracking using small-scale automated vehicles.

\subsection{Outline and Notation}
This paper is structured as follows. Section~\ref{sec:problem} introduces the objectives and problem formulation, followed by a detailed description of the MIP-DM method in Section~\ref{sec:MIQP_form}. The embedded MIQP solver is described in Section~\ref{sec:MIMPC_solver}, and the simulation results are shown in Section~\ref{sec:simulation}. Finally, Section~\ref{sec:results} presents results from the hardware experiments and our conclusions are established in Section~\ref{sec:concl}.

{\em Notation}: 
$\R$, $\R_+$, $\R_{0+}$ ($\Z$, $\Z_+$, $\Z_{0+}$) are the set of real, positive real and nonnegative real (integer) numbers, $\B=\{0,1\}$, and $\Z_a^b = \{a, a+1, \ldots, b-1, b\}$. The logical operators {\em and}, {\em or}, {\em xor}, {\em not} are $\land$, $\lor$, $\lxor$, $\lnot$, and the logical operators {\em implies} and {\em equivalent (if and only if)} are $\implies$, $\iff$. Inequalities between vectors are intended componentwise.

\section{Problem Setup and Formulation}
\label{sec:problem}

\begin{figure}
	\centering
	\begin{minipage}{.52\linewidth}%
		\begin{subfigure}[b]{\textwidth}
			\centering
			\includegraphics[width=0.83\textwidth,trim={0cm 0cm 0cm 0cm},clip]{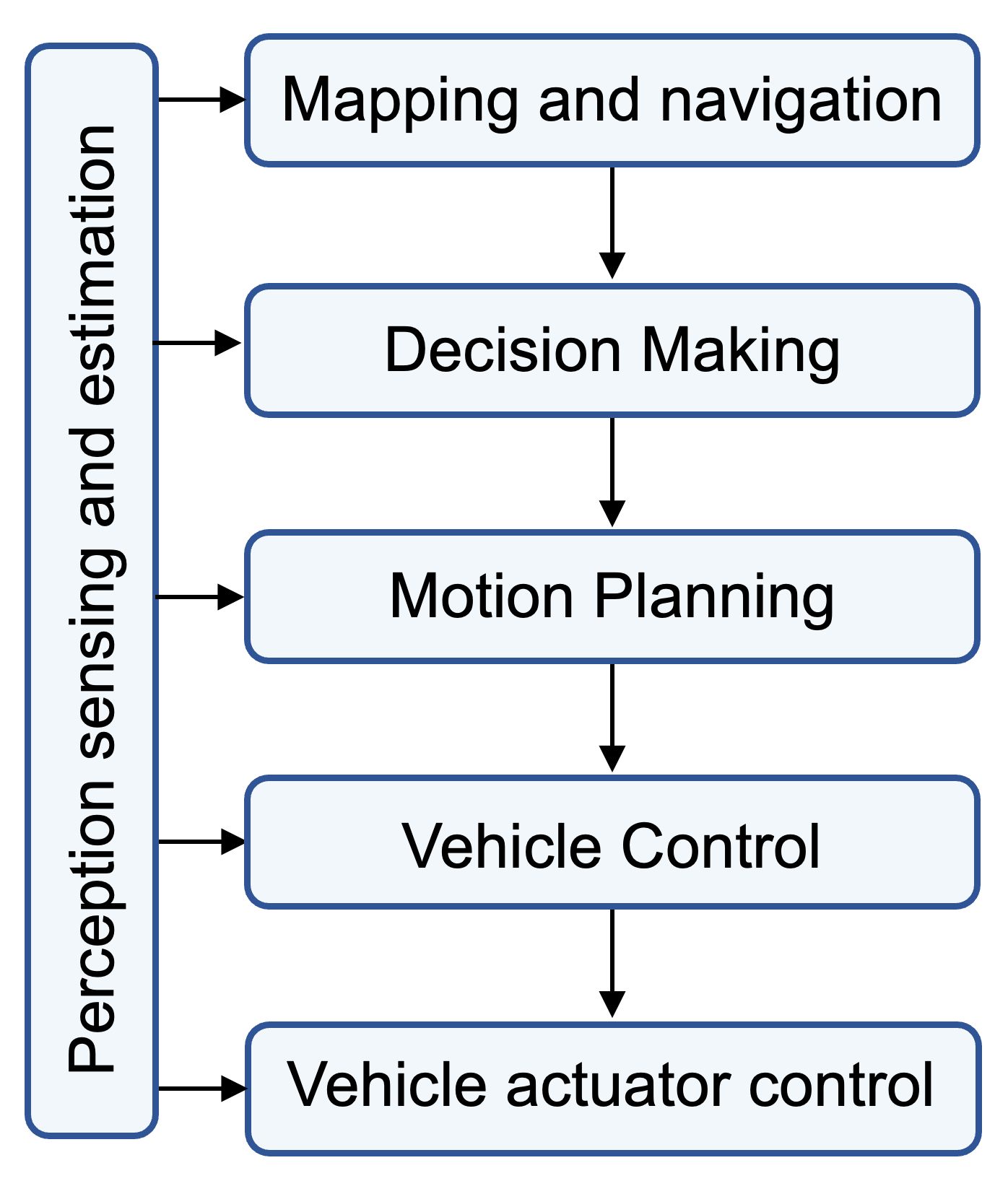}
			\captionsetup{font=normalsize,labelfont={sf}}
			\caption[]%
			{\normalsize Typical architecture, e.g.~\cite{Ahn2021}.}
			\label{fig:MIP_DM_architecture1}
		\end{subfigure}
	\end{minipage}
	\begin{minipage}{.46\linewidth}%
		\begin{subfigure}[b]{\textwidth}  
			\centering
			\includegraphics[width=0.93\textwidth,trim={0cm 0cm 0cm 0cm},clip]{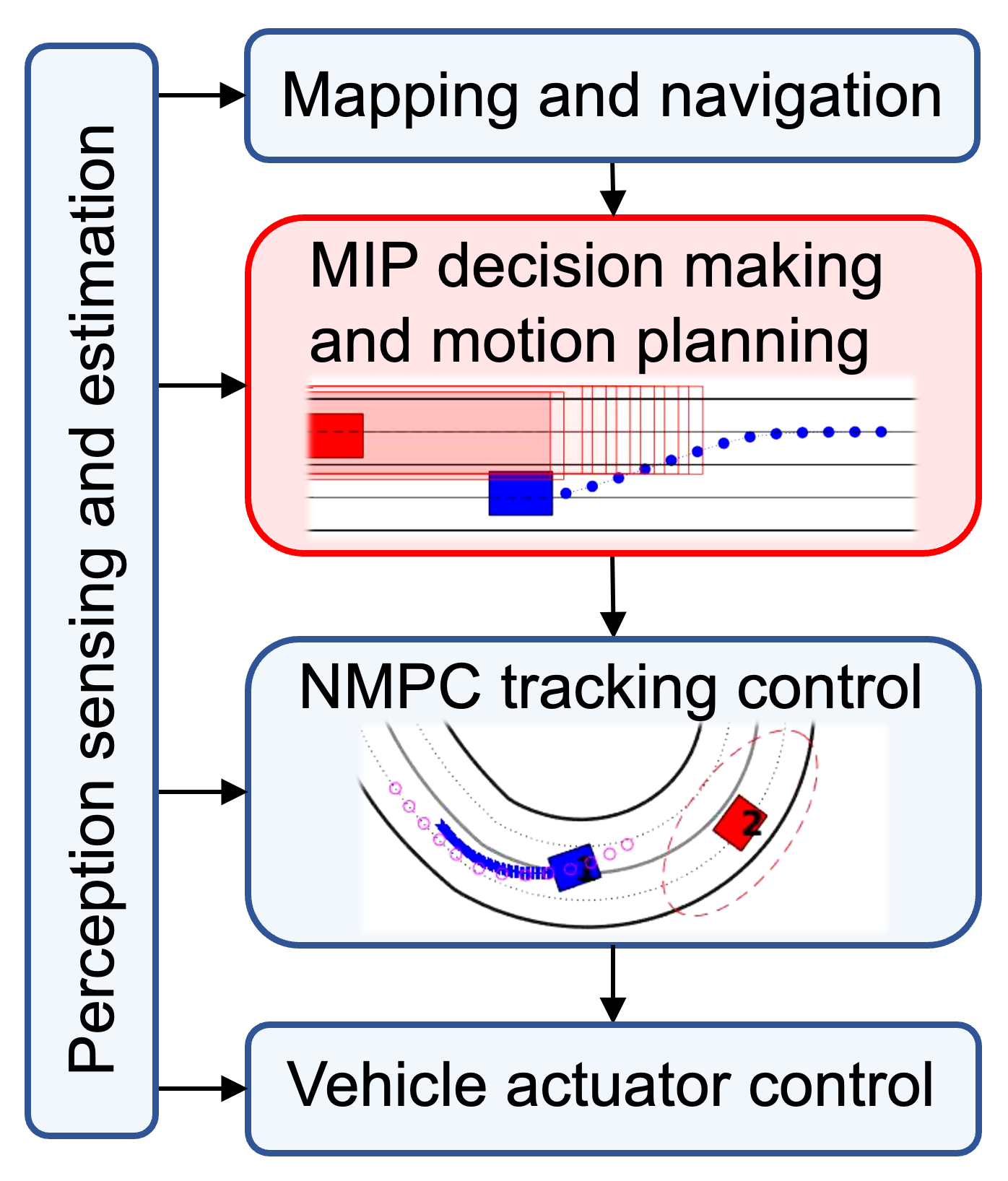}
			\captionsetup{font=normalsize,labelfont={sf}}
			\caption[]%
			{\normalsize MIP-DM architecture.}
			\label{fig:MIP_DM_architecture2}
		\end{subfigure}
	\end{minipage}
	\caption[ ]{Multi-layer control architecture for ADAS/AD.}
	\label{fig:MIP_DM_architecture}
\end{figure}

This section briefly describes common components in a multi-layer guidance and control architecture for ADAS/AD, \change{and then introduces the MIOCP formulation for MIP-DM}.

\subsection{Multi-layer Control Architecture for Automated Driving}

A typical guidance and control architecture is illustrated in Figure~\ref{fig:MIP_DM_architecture1}. 
A perception, sensing and estimation module uses various on-board sensor information, such as radar, LIDAR, camera, and global positioning system~(GPS) information, to estimate the vehicle states, parameters, and parts of the surroundings relevant to the driving scenario~\cite{Brummelen2018}.
Based on a route given by a navigation system, a decision making module \change{determines what maneuvers to perform, e.g.,} lane changing, stopping, waiting, intersection crossing~\cite{Ahn2021}. Then, a motion planning system generates a collision-free and kinematically feasible trajectory to perform the maneuvers, see, e.g.,~\cite{Berntorp2018c}. A vehicle control system computes the input signals to execute the motion planning trajectory, see, e.g.,~\cite{Quirynen2020}. Additional low-level controllers \change{operate the vehicle actuators}.

\subsection{Setup for MIP-based Decision Making~(MIP-DM)}

In this paper, an autonomous vehicle must reach a desired destination while obeying the traffic rules. This requires the vehicle to adjust its velocity to obey the speed limits, to avoid collisions, to follow and change lanes, and to cross intersections following right of way rules. We propose an alternative architecture \change{to that in Fig.~\ref{fig:MIP_DM_architecture1}}, using MIP-based vehicle decision making, see Fig.~\ref{fig:MIP_DM_architecture2}. The problem setup in this work requires the following simplifying assumptions.
\begin{Assumption}
	\change{There exists a prediction time window along which the following are known}
	\begin{enumerate}
		\item the position and orientation for each of the obstacles in a \change{sufficiently large} neighborhood of the ego vehicle,
		\item the map information, including center lines, road curvature and lane widths within the current road segment,
		\item the current traffic rules and any changes to the rules, e.g., traffic light timings and/or speed zone changes. \hfill\QEDopen
\end{enumerate}
\label{as:problem}
\end{Assumption}

Assumption~\ref{as:problem}.1 requires the vehicle to be equipped with sensors to detect static and dynamic obstacles within a given range and to locate itself in the environment. Furthermore, the vehicle must be equipped with a module that provides conservative predictions for future trajectories of the dynamic obstacles, 
e.g., using techniques referenced in~\cite{Paden2016,Guanetti2018}.
Assumption~\ref{as:problem}.2 requires the availability of map information and/or the use of online updates and corrections to such map information~\cite{Brummelen2018}. Assumption~\ref{as:problem}.3 requires a combination of map information, online perception~\cite{Brummelen2018}, and/or vehicle-to-infrastructure~(V2I) communication~\cite{ersal2020}. Based on these assumptions, we define the problem statement and objectives.

\begin{Definition}[MIP Decision Making~(MIP-DM)]\label{def:DMandMP} Under Assumption~\ref{as:problem} and given navigation information, at each sampling instant, \change{the MIP-DM} module solves an MIOCP on embedded hardware and under strict timing requirements. The solution provides desired maneuvers that the vehicle should execute, and a coarse trajectory, i.e., a sequence of waypoints and target velocities, over a horizon of several seconds for the vehicle control module to execute the \change{maneuver}. \hfill\QEDopen
\end{Definition}

\change{Based on Def.~\ref{def:DMandMP},} the trajectory computed by the MIP-DM is executed by a vehicle control module, e.g., the NMPC reference tracking controller in Fig.~\ref{fig:MIP_DM_architecture2}.

\subsection{\change{Mixed-integer Optimal Control Problem~(MIOCP)}}

At each sampling time instant, the proposed MIP-DM solves the following MIOCP
\begin{subequations} \label{eq:MIOCP}%
	\begin{alignat}{5}
		\hspace{-3mm}\underset{X,\,U}{\text{min}} \quad &\sum_{i=0}^{N} \frac{1}{2}\begin{bmatrix}	x(i) \\ u(i)	\end{bmatrix}^\top \!H(i) \begin{bmatrix}	x(i) \\ u(i)	\end{bmatrix} + \begin{bmatrix}	q(i) \\ r(i)	\end{bmatrix}^\top \begin{bmatrix}	x(i) \\ u(i)	\end{bmatrix} && \label{OCP:obj}\\
		\hspace{-3mm}\text{s.t.} \quad\; 
		& x(i+1) = \begin{bmatrix} A(i) &\hspace{-1mm} B(i) \end{bmatrix} \begin{bmatrix} x(i) \\ u(i) \end{bmatrix} + a(i), && \hspace{-2em} \forall i \in \Z_{0}^{N-1}, \label{OCP:dyn}\\
		& \begin{bmatrix} \ubar{x}(i) \\ \ubar{u}(i) \end{bmatrix} \leq \begin{bmatrix} x(i) \\ u(i) \end{bmatrix} \leq \begin{bmatrix} \bar{x}(i) \\ \bar{u}(i) \end{bmatrix}, \quad &&\hspace{-2em} \forall i \in \Z_{0}^{N}, \label{OCP:bounds} \\
		& \blo(i) \leq \begin{bmatrix} C(i) &\hspace{-1mm} D(i) \end{bmatrix} \begin{bmatrix} x(i) \\ u(i) \end{bmatrix} \leq \bup(i), \quad &&\hspace{-2em} \forall i \in \Z_{0}^{N}, \label{OCP:ineq} \\
		& \change{u_j(i)} \in \Z, \quad \forall j \in \I(i), && \hspace{-2em} \forall i \in \Z_{0}^{N}, \label{OCP:int}
	\end{alignat}
\end{subequations} 
where \change{$i \in \{0,1,\ldots,N\}$ is the time, $N$ is the horizon length,} the state variables are $x(i) \in \R^{\nx^i}$, the control \change{and auxiliary} variables are $u(i) \in \R^{\nU^i}$ and $\I(i)$ denotes the index set of integer decision variables, i.e., the cardinality $| \I(i) | \le \nU^i$ denotes the number of integer variables at each time step.
The objective in~\eqref{OCP:obj} defines a linear-quadratic function with positive semi-definite Hessian matrix $H(i) \succeq 0$ and gradient vectors $q(i) \in \R^{\nx^i}$ and $r(i) \in \R^{\nU^i}$. The constraints include dynamic constraints in~\eqref{OCP:dyn}, simple bounds in~\eqref{OCP:bounds}, affine inequality constraints in~\eqref{OCP:ineq} and integer feasibility constraints in~\eqref{OCP:int}. The initial state constraint \change{$x(0) = \hat{x}_t$}, where \change{$\hat{x}_t$} is a current state estimate \change{at time $t$}, can be enforced using the simple bounds in~\eqref{OCP:bounds}. The MIOCP~\eqref{eq:MIOCP} includes control variables on the terminal stage, $u(N) \in \R^{\nU^N}$, \change{due to possibly needing} auxiliary variables to formulate the mixed-integer inequality constraints. A binary optimization variable $\change{u_j(i)} \in \{0,1\}$ can be defined as an integer variable $\change{u_j(i)} \in \Z$ in~\eqref{OCP:int}, including the simple bounds $0 \le \change{u_j(i)} \le 1$ in~\eqref{OCP:bounds}. 
\change{For compactness, we denote $X=[x(0)^\top,\ldots, x(N)^\top]^\top$ and $U=[u(0)^\top,\ldots, u(N)^\top]^\top$.}
\change{The MIOCP~\eqref{eq:MIOCP} can be reformulated as a block-sparse structured MIQP~\cite{Quirynen2022}, and solved with corresponding algorithms.}

\section{Mixed-Integer Quadratic Programming for Vehicle Decision Making and Motion Planning}
\label{sec:MIQP_form}

\change{Next, we describe the MIP-DM for achieving safe and real-time feasible automated driving in real-world scenarios.}

\subsection{Linear Vehicle Model in Road-aligned Frame}
\label{sec:model}

The curvilinear coordinate system used in the prediction model of the MIOCP~\eqref{eq:MIOCP} is shown in Fig.~\ref{fig:curvlinear}. A similar coordinate system has been used for predictive control, e.g.,~\cite{Frasch2013b,Gao2012SpatialPC}. The vehicle position is described by $(\ps,\pn)$, where $\ps$ denotes the progress along the center line of the lane in which the ego vehicle is driving, and $\pn$ denotes the normal distance of the vehicle position from the center line. 

\begin{figure}[h]
	\centerline{\hbox{
			\includegraphics[width=0.38\textwidth,trim={1cm 0cm 1cm 1cm},clip]{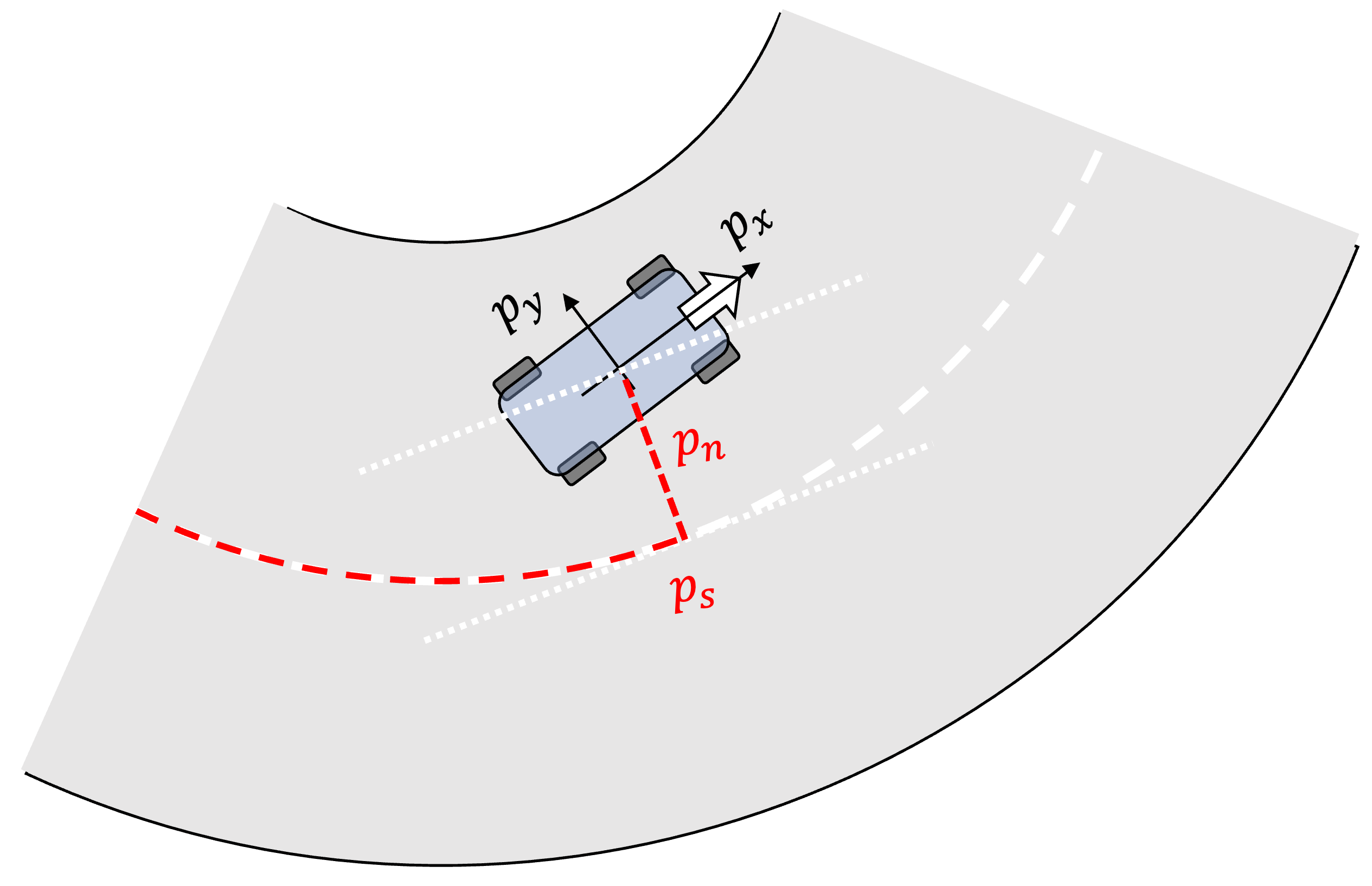}}}		
	\caption{Road-aligned curvilinear coordinate system for a curved segment; $\ps$ is the arc length along the center line and $\pn$ is the lateral deviation.}
	\label{fig:curvlinear}
\end{figure}

\begin{Assumption} \label{as:modeling}
The turning radius is much larger than the wheelbase of the vehicle, such that the steering and slip angles are relatively small and their difference for the outside and inside wheels is negligible. \hfill\QEDopen
\end{Assumption}

\change{Based on Ass.~\ref{as:modeling}, which is common in vehicle motion planning~\cite{Berntorp2018,Berntorp2018c}, we use a}
simplified linear vehicle model in the curvilinear coordinate system and with decoupled longitudinal and lateral kinematics
\begin{equation}
\begin{aligned}
\ps(i+1) &= \ps(i) + \Ts \, \vs(i), \\
\vs(i+1) &= \vs(i) + \Ts \, \as(i), \\
\pn(i+1) &= \pn(i) + \Ts \, \vn(i), \\
\end{aligned}  \label{eq:kinematics}
\end{equation}
where the control inputs are the longitudinal acceleration $\as(i)$ and the lateral velocity $\vn(i)$ at each time step $i \in \Z_{0}^{N-1}$.
To approximate the nonholonomic constraints of Ackerman steering for vehicles, we enforce the linear inequality constraint on the lateral and longitudinal velocity
\begin{equation}
-\alpha\, \vs(i) \leq \vn(i) \leq \alpha\, \vs(i), \quad i \in \Z_{0}^{N-1}, \label{eq:steering}
\end{equation}
where $\alpha > 0$, and we assume $\vs(i) \ge 0$ at all time steps. 

\begin{Proposition}
\label{prop:steering}
The inequality constraint in~\eqref{eq:steering} is a linear approximation of a vehicle steering limit and, using a kinematic bicycle model, 
\begin{equation}
\alpha = \text{sin}\left( \text{tan}^{-1}\left( \frac{l_{\mathrm{r}}}{R^{\mathrm{min}}} \right) \right) \approx \frac{l_{\mathrm{r}}}{R^{\mathrm{min}}}, \label{eq:alpha}
\end{equation}
where $l_{\mathrm{r}}$ denotes the distance from center of gravity to the rear axle and $R^{\mathrm{min}}$ denotes the vehicle's minimum turning radius.
\end{Proposition}
\begin{proof}
Considering the kinematic bicycle model~\cite{Rajamani2012}
\begin{subequations}
\begin{alignat}{5}
	\dot{p}_{\mathrm{X}} &= v\, \text{cos}(\psi+\beta), &\quad \dot{p}_{\mathrm{Y}} &= v\, \text{sin}(\psi+\beta), \\
	\dot{\psi} &= v\, \frac{\text{cos}(\beta)}{L} \text{tan}(\delta), &\quad \beta &= \text{tan}^{-1} \left( \frac{l_{\mathrm{r}}\, \text{tan}(\delta)}{L} \right) ,
\end{alignat} \label{eq:nonl_kinematics}%
\end{subequations}
where $(\pX, \pY)$ is the position of the vehicle's center of gravity in an absolute frame, and $L = l_{\mathrm{f}}+l_{\mathrm{r}}$ is the wheelbase. For a constant radius $R$, or road curvature $\frac{1}{R}$, the yaw rate is $\dot{\psi} = \frac{v}{R}$~\cite[Sec.~2.2]{Rajamani2012}, such that $\text{tan}(\delta) \approx \frac{L}{R}$ and $\beta = \text{tan}^{-1} \left( \frac{l_{\mathrm{r}}}{R} \right)$. We know that the lateral velocity is $\dot{p}_{\mathrm{y}} = v\, \text{sin}(\beta)$ in the car body frame.
Given a minimum turning radius $R^{\mathrm{min}}>0$, the steady state lateral velocity is $\vy^{\mathrm{max}} = v\, \text{sin}(\text{tan}^{-1}( \frac{l_{\mathrm{r}}}{R^{\mathrm{min}}}))$, and therefore $\alpham = \text{sin}(\text{tan}^{-1}(\frac{l_{\mathrm{r}}}{R^{\mathrm{min}}})) \approx \frac{l_{\mathrm{r}}}{R^{\mathrm{min}}} >0$ in~\eqref{eq:steering}.
\end{proof}

The vehicle model~\eqref{eq:kinematics} is an approximation of more precise models, see, e.g.,~\cite{Berntorp2014a}, which are usually nonlinear. However, \change{the MIOCP~\eqref{eq:MIOCP} provides a reference trajectory for the vehicle controller and operates in normal driving conditions when some of} the vehicle nonlinearities, such as the road-tire friction curve, are not excited, while others can be neglected because the decision-making operates over long horizons with a fairly coarse sampling period. Modeling errors are compensated by the vehicle control layer as illustrated in Fig.~\ref{fig:MIP_DM_architecture}.

\begin{Remark}
\label{rem:steering}
Given a time varying road radius $R(i)$, which may be positive or negative depending on the direction of the road curvature, the lateral velocity in~\eqref{eq:kinematics} is bounded as $\vn(i) \le \vy^{\mathrm{max}} - \vy^{\mathrm{R}}(i)$, where $\vy^{\mathrm{max}} = v \, \alpham$, and $\vy^{\mathrm{R}}(i) = v \, \alphar(i)$ denotes the steady state lateral velocity to follow the center of the road with radius $R(i)$. Eq.~\eqref{eq:steering} may be replaced by 
\begin{equation}
\hspace{-5mm}	\left(-\alpham - \alphar(i)\right) \, \vs(i) \leq \vn(i) \leq \left(\alpham - \alphar(i)\right) \, \vs(i), \label{eq:steering_new}
\end{equation}
where $\alpham = \frac{l_{\mathrm{r}}}{R^{\mathrm{min}}} > 0$ defines the maximum steering and $\alphar(i) = \frac{l_{\mathrm{r}}}{R(i)}$ defines the steering needed to follow the center of the road with radius $R(i)$, following Proposition~\ref{prop:steering}. \hfill\QEDopen
\end{Remark}

\begin{Remark}
Proposition~\ref{prop:steering} uses a simple approximation of the steady-state cornering equations in~\cite[Sec.~3.3]{Rajamani2012}.
Alternatively, the cornering equations could be directly used to compute a time-varying value for $\alpha(i)$ that depends on the predicted velocity and the road curvature. \hfill\QEDopen
\end{Remark}

\subsection{Lane Change and Timing Delay Constraints}
\label{sec:laneChange}

\change{We enforce} lane bound constraints
\begin{equation}
-\frac{\wl}{2} \le \pn(i) - \pref(i) \le \frac{\wl}{2}, \quad i \in \Z_{0}^{N}, \label{eq:lane_bound}
\end{equation}
where $\wl$ denotes a lane width given by the map and $\pref \in \R$ is an auxiliary state variable that denotes the lateral position of the center line of the current lane of the vehicle. \change{For equal lane width values $\wl$}, the vehicle is in lane $j$ if $\pref = (j-1)\,\wl$ for $j \in \{1,\ldots, \nl\}$, where $\nl$ is the number of lanes in the current traffic scenario. Even though the reference lane value may jump from one time step $\pref(i)$ to the next $\pref(i+1)$, it may take multiple time steps for the lateral position to transition from the center line of one lane to the next, i.e., $\pn(i-l) \approx \pref(i)$ and $\pn(i+k) \approx \pref(i+1)$, where $l \ge 0$ and $k \ge 1$. 

\subsubsection{Lane Change Decision Constraints}
We use two binary variables $\lu(i), \ld(i) \in \{0, 1\}$ that denote whether the vehicle performs a lane change \change{left or right}, respectively, at time step $i \in \Z_{0}^{N-1}$.
We also introduce an auxiliary variable $\pc \in \R$ defined by $\lu(i)$, $\ld(i)$ through
\begin{equation}
\begin{aligned}
\lu(i) = 1 &\implies \pc(i) = \wl \;\land\; \ld(i)=0, \\
\ld(i) = 1 &\implies \pc(i) = -\wl \;\land\; \lu(i)=0, \\
\hspace{-2mm} \lu(i) = 0 \;\land\; \ld(i) = 0 &\implies \pc(i) = 0.
\end{aligned} \label{eq:lc_implication}
\end{equation}
\change{For $i \in \Z_{0}^{N-1}$, the implications in~\eqref{eq:lc_implication} may be} implemented as 
\begin{subequations}
\begin{alignat}{3}
-\wl \, (\lu(i) + \ld(i)) &\le \pc(i) \le \wl \, (\lu(i) + \ld(i)),  \label{eq:lane_change1} \\
-\wl + 2\,\wl \, \lu(i) &\le \pc(i) \le \wl - 2\,\wl \, \ld(i).  \label{eq:lane_change2}
\end{alignat} \label{eq:lane_change}%
\end{subequations}
Constraint~\eqref{eq:lane_change2} ensures that $\lu(i) + \ld(i) \le 1$.
\change{The auxiliary state dynamics are}
\begin{subequations}
\begin{alignat}{5}
\pref(i+1) &= \pref(i) + \pc(i), \label{eq:Yref_dyn} \\
n_{\mathrm{LC}}(i+1) &= n_{\mathrm{LC}}(i) + (\lu(i) + \ld(i)), \label{eq:LC_count}
\end{alignat} \label{eq:aux_state}%
\end{subequations}
where 
$n_{\mathrm{LC}}(i)$ counts the number of lane changes over the prediction horizon \change{and is initialized to} $n_{\mathrm{LC}}(0) = 0$.

\begin{Remark}
The state $n_{\mathrm{LC}}(i) \in \Z$ is an integer variable, but it can be relaxed to be continuous because the sum in~\eqref{eq:LC_count} is guaranteed to be integer. Similarly, $\pref$ and $\pc$ could be reformulated as $\pref = \wl \, \tilde{p}_{\mathrm{n}}^{\mathrm{ref}}$ and $\pc = \wl \, \tilde{\Delta}_{\mathrm{c}}$, where $\tilde{p}_{\mathrm{n}}^{\mathrm{ref}} \in \{ 0,1,\ldots,\nl-1 \}$ and $\tilde{\Delta}_{\mathrm{c}} \in \{ -1,0,1 \}$.
State of the art MIP solvers can possibly use these integer feasibility constraints to reduce the computational effort~\cite{nemhauser1988integer}. For simplicity, we only use continuous and binary optimization variables. 
\hfill\QEDopen
\end{Remark}

\subsubsection{Timing Delay Constraints for Lane Changes}

We enforce a minimum time delay of $t_{\mathrm{min}}$ between two consecutive lane changes.
The lane change variables $\lu(i),\ld(i) \in \{0,1\}$ reset a timer $\tc(i)$ as
\begin{equation}
\tc(i+1) = \left\{
\begin{array}{ll}
\tc(i) + \Ts &\quad \text{if} \;\; \lu(i) = \ld(i) = 0, \\
0 &\quad \text{otherwise},
\end{array}
\right. \label{eq:contact}
\end{equation}
which can be implemented by constraints
\begin{equation}
\begin{aligned}
-(1-\lc(i)) \, M &\leq \tc(i+1) \leq (1-\lc(i)) \, M, \\
\hspace{-2mm} \tc(i) + \Ts  - \lc(i) \, M &\leq \tc(i+1) \leq \tc(i) + \Ts  + \lc(i) \, M,
\end{aligned} \label{eq:timer}
\end{equation}
where $\lc(i) = \lu(i)+\ld(i)$ is a compact notation, and $M \gg 0$ is a large positive constant in a big-M formulation~\cite{nemhauser1988integer}. Given $\tc(i)$, we impose a minimum time between lane changes
\begin{equation}
t_{\mathrm{min}} - M\, (1-\lu(i)-\ld(i)) \le \tc(i), \; i \in \Z_{0}^{N-1}, \label{eq:min_time}
\end{equation}
i.e., $\lu(i) = 1$ or $\ld(i) = 1$ only if $\tc(i) \ge t_{\mathrm{min}}$. \change{In a receding horizon implementation of the MIP-DM, the timer $\tc(0)$ is initialized to the value from the previous time step.}

\subsection{Polyhedral Obstacle Avoidance Constraints}
\label{sec:obstacles}

The MIP-DM enforces obstacle avoidance constraints to avoid a \change{region of} collision risk around other traffic participants, e.g., vehicles, bicycles or pedestrians. The \change{position and dimensions} of the safety region may be time varying and adapted to a prediction of the behavior for each of the traffic participants. In addition, obstacle avoidance constraints enforce stopping maneuvers, e.g., in case of a stop sign or a red traffic light at an intersection. Per Assumption~\ref{as:problem}, the prediction of obstacle motions, the map information and the traffic rules are known.
For simplicity, we use axis-aligned rectangular collision regions, as illustrated in Figure~\ref{fig:obstacle_avoidance}.
Alternatively, any polyhedral representation of the collision regions could be used, see, e.g.,~\cite{Richards2005}. The size of the collision region around the obstacle is increased with the geometric shape of the ego vehicle and includes an additional safety margin for robustness to discretization errors, model mismatch and/or disturbances.

As shown in Fig.~\ref{fig:obstacle_avoidance}, obstacle avoidance for an axis-aligned rectangular region results in four disjoint feasible sets. We introduce $4$ auxiliary binary variables $\delta_{\mathrm{o}}^j(i) = [ \delta_{\mathrm{o},k}^j(i) ]_{k \in \Z_1^4}$ for $j \in \Z_1^{\nobs}$, to implement the logical implications
\begin{equation}
\begin{aligned}
	\hspace{-2mm} \delta_{\mathrm{o},1}^j = 1 &\iff \ps \le \ups^j + \scolx,  \\
	\hspace{-2mm}  \delta_{\mathrm{o},2}^j = 1 &\iff \ps \ge \bps^j - \scolx, \\
	\hspace{-2mm}  \delta_{\mathrm{o},3}^j = 1 &\implies \ups^j + \scolx \le \ps \le \bps^j - \scolx \; \land \; \pn \le \upn^j + \scoly, \\
	\hspace{-2mm}  \delta_{\mathrm{o},4}^j = 1 &\implies \ups^j + \scolx \le \ps \le \bps^j - \scolx \; \land \; \pn \ge \bpn^j - \scoly,
\end{aligned} \label{eq:col_implications}
\end{equation}
where we omit the index $i \in \Z_{0}^{N}$ for readability, and we use slack variables $\scolx(i) \ge 0$, $\scoly(i) \ge 0$ to ensure feasibility.
\change{We impose that} the ego vehicle is in one of the feasible sets \change{by} $\sum_{k=1}^4 \delta_{\mathrm{o},k}^j(i) = 1$. Hard obstacle avoidance constraints can be defined by enforcing upper bounds on the slack variables $0 \le  \scolx(i) \le \bar{\nu}_{\mathrm{s}}^{\mathrm{c}}$ and $0 \le  \scoly(i) \le \bar{\nu}_{\mathrm{n}}^{\mathrm{c}}$, see Fig.~\ref{fig:obstacle_avoidance}. To reduce the number of variables in the MIP formulation, a single slack variable $\scolx(i) = a_{\mathrm{sn}}\, \scoly(i)$ may be used, where $a_{\mathrm{sn}} > 0$ is a constant. The implications in~\eqref{eq:col_implications} can be \change{implemented as}
\begin{equation}
\begin{aligned}
\ups^j(i) + \scolx(i) \le \ps(i) + \M \delta_{\mathrm{o},1}^j(i) &\le \ups^j(i) + \scolx(i) + \M, \\
\bps^j(i) - \scolx(i) - \M \le \ps(i) - \M \delta_{\mathrm{o},2}^j(i) &\le \bps^j(i) - \scolx(i), \\
\pn(i) + \M \delta_{\mathrm{o},3}^j(i) &\le \upn^j(i) + \scoly(i) + \M, \\
\pn(i) - \M \delta_{\mathrm{o},4}^j(i) &\ge \bpn^j(i) - \scoly(i) - \M, \\
\ps(i) + \M (\delta_{\mathrm{o},3}^j(i) + \delta_{\mathrm{o},4}^j(i)) &\le \bps^j(i) - \scolx(i) + \M, \\
\ps(i) - \M (\delta_{\mathrm{o},3}^j(i) + \delta_{\mathrm{o},4}^j(i)) &\ge \ups^j(i) + \scolx(i) - \M, \\
\sum_{k=1}^4 \delta_{\mathrm{o},k}^j(i) &= 1,
\end{aligned} \label{eq:obstacle_constraints_v4}
\end{equation}
where $M \gg 0$ denotes \change{the big-M constant}.

\begin{figure}[t]
\centerline{\hbox{
		\includegraphics[width=0.38\textwidth,trim={0cm 0cm 0cm 0cm},clip]{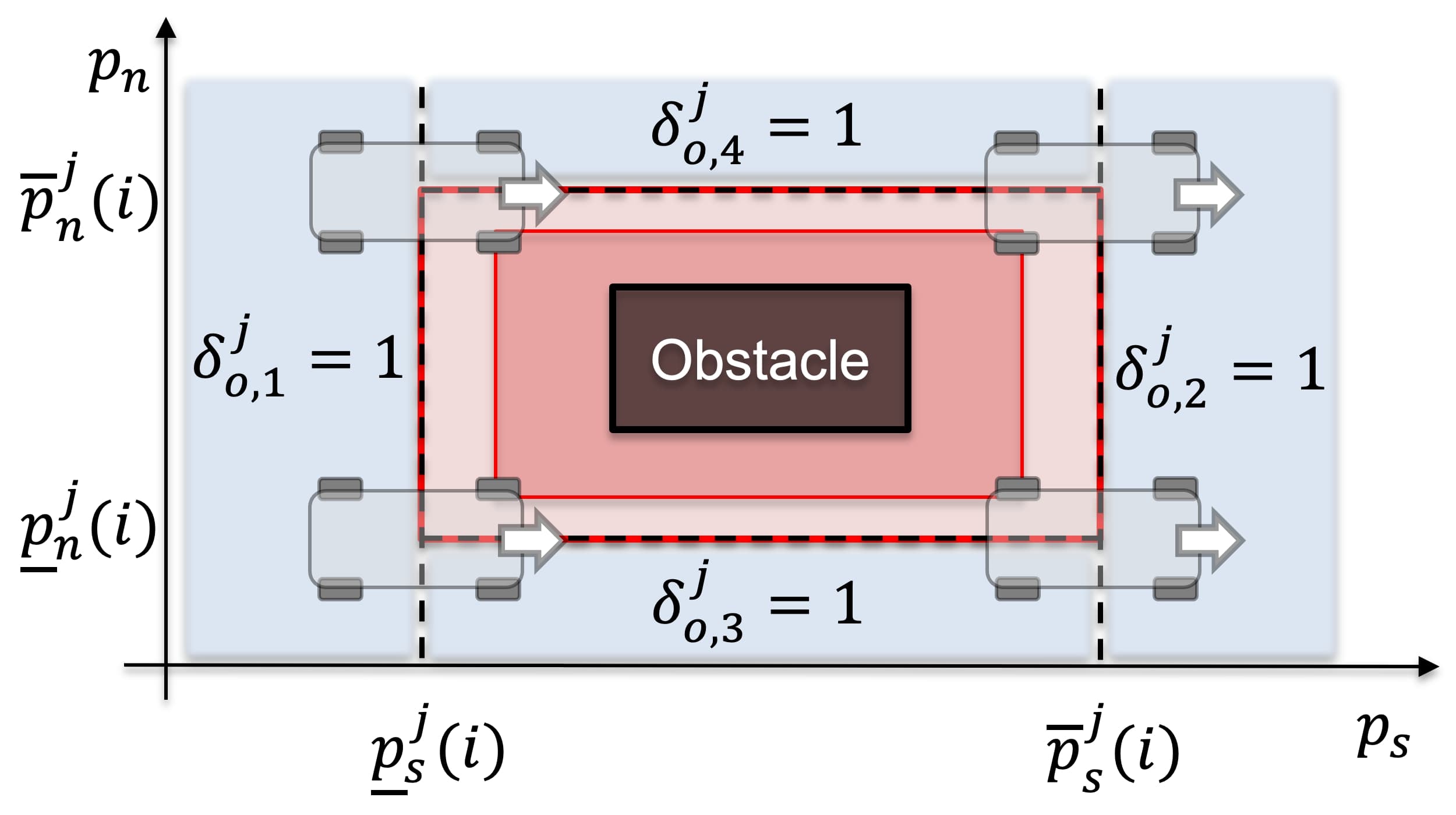}}}		
\caption{Obstacle avoidance constraints using binary variables and an axis-aligned rectangular collision region. The extent of the region is increased by the geometric shape of the ego vehicle and includes an additional safety margin. The light red shaded region is defined by soft constraints, while the dark region is defined by hard constraints.}
\label{fig:obstacle_avoidance}
\end{figure}

\begin{Remark}
	For each obstacle $j \in \Z_1^{\nobs}$ in~\eqref{eq:obstacle_constraints_v4}, we predict its position based on a constant velocity profile in curvilinear coordinates. Future work may include the use of a more advanced prediction model, e.g., a switching dynamical model~\cite{Suriyarachchi2023} or a neural network classifier~\cite{Chen2021}. \hfill\QEDopen
\end{Remark}

\subsubsection{Traffic Intersection Crossing Constraints}
The obstacle avoidance constraints in~\eqref{eq:obstacle_constraints_v4} are also used to prevent the ego vehicle from crossing a traffic intersection, e.g., forcing the vehicle to stop during a particular time window.
Similar to Fig.~\ref{fig:obstacle_avoidance}, the avoidance region is defined by the dimensions of the intersection, enlarged to account for the physical shape of the ego vehicle and with additional safety margins to account for modeling errors. If the intersection is controlled by traffic lights and if the traffic light changes are known, e.g., using V2I communication~\cite{ersal2020}, the intersection crossing constraints are time-varying within the prediction horizon. For example, if it is known that a traffic light will turn red, the intersection crossing constraints~\eqref{eq:obstacle_constraints_v4} cause the ego vehicle to slow down and plan a stopping maneuver. Similarly, the constraints are relaxed at future time steps within the prediction horizon when the traffic lights are predicted to become green. Alternatively, 
the intersection crossing constraints may be implemented based on map information and/or the perception system~\cite{Brummelen2018}.

\subsection{Zone-dependent Traffic Rules}
\label{sec:zones}

In real-world scenarios, traffic rules may change when the vehicle transitions into a particular zone. From one zone to the next, following traffic rule constraints may change 
\begin{itemize}
	\item speed limit, e.g., the vehicle entering a low-speed zone,
	\item allowed lane changes, e.g., when no lane changes are allowed inside a particular zone,
	\item available lanes, e.g., when a three-lane road transitions into a two-lane road or when the vehicle must merge.
\end{itemize}

We introduce binary variables $\delta_{\mathrm{z}} = [\delta_{\mathrm{z}}^1, \ldots, \delta_{\mathrm{z}}^{\nz}]$, where $\nz$ denotes the number of position-dependent zones. Each zone is represented by a range $[\ubar{p}_j,\bar{p}_j]$ for $j \in \Z_1^{\nz}$ in the longitudinal $\ps$-direction. We detect whether the vehicle is in zone $j$ as
\begin{equation*}
	\delta_{\mathrm{z}}^j(i) = 1 \quad \Rightarrow \quad \ubar{p}_j(i) \le \ps(i) \le \bar{p}_j(i),
\end{equation*}
which can be implemented as
\begin{equation}
\begin{aligned}
\ubar{p}_j(i) - \M (1-\delta_{\mathrm{z}}^j(i)) \le \ps(i) &\le \bar{p}_j(i) + \M (1-\delta_{\mathrm{z}}^j(i)). \label{eq:zones}
\end{aligned}
\end{equation}
Because the position-dependent zones are disjoint, the vehicle needs to be inside exactly one zone, i.e., $\sum_{j=1}^{\nz} \delta_{\mathrm{z}}^j = 1$.

The auxiliary binary variables $\delta_{\mathrm{z}}$ and constraints in~\eqref{eq:zones} enable implementing the zone-dependent traffic rules. For example, changing speed limits can be enforced by
\begin{equation}
\vs(i) \leq \sum_{j=1}^{\nz} \delta_{\mathrm{z}}^j \, \bar{v}_{\mathrm{s}}^j(i), \label{eq:newVelConstraint}
\end{equation}
where the speed limit $\bar{v}_{\mathrm{s}}^j(i)$ corresponds to zone $j=1,\ldots,\nz$ and $\sum_{j=1}^{\nz} \delta_{\mathrm{z}}^j = 1$.
Similarly, the allowed number of lane changes can be adjusted as
\begin{equation}
	n_{\mathrm{LC}}(i) \leq \sum_{j=1}^{\nz} \delta_{\mathrm{z}}^j \, \bar{n}^j_{\mathrm{LC}}, \label{eq:numberLC}
\end{equation}
and the constraints on feasible lanes can be adjusted as
\begin{equation}
\sum_{j=1}^{\nz} \delta_{\mathrm{z}}^j  \, \ubar{p}_{\mathrm{n}}^{\mathrm{ref},j}(i) \leq p_{\mathrm{n}}^{\mathrm{ref}}(i) \leq \sum_{j=1}^{\nz} \delta_{\mathrm{z}}^j \, \bar{p}_{\mathrm{n}}^{\mathrm{ref},j}(i). \label{eq:laneBounds}
\end{equation}
Figure~\ref{fig:zones_figure} shows the transition from a three-lane road segment into a two-lane road segment using~\eqref{eq:laneBounds}.

\begin{figure}[t]
	\centerline{\hbox{
			\includegraphics[width=0.36\textwidth,trim={0cm 0cm 0cm 0cm},clip]{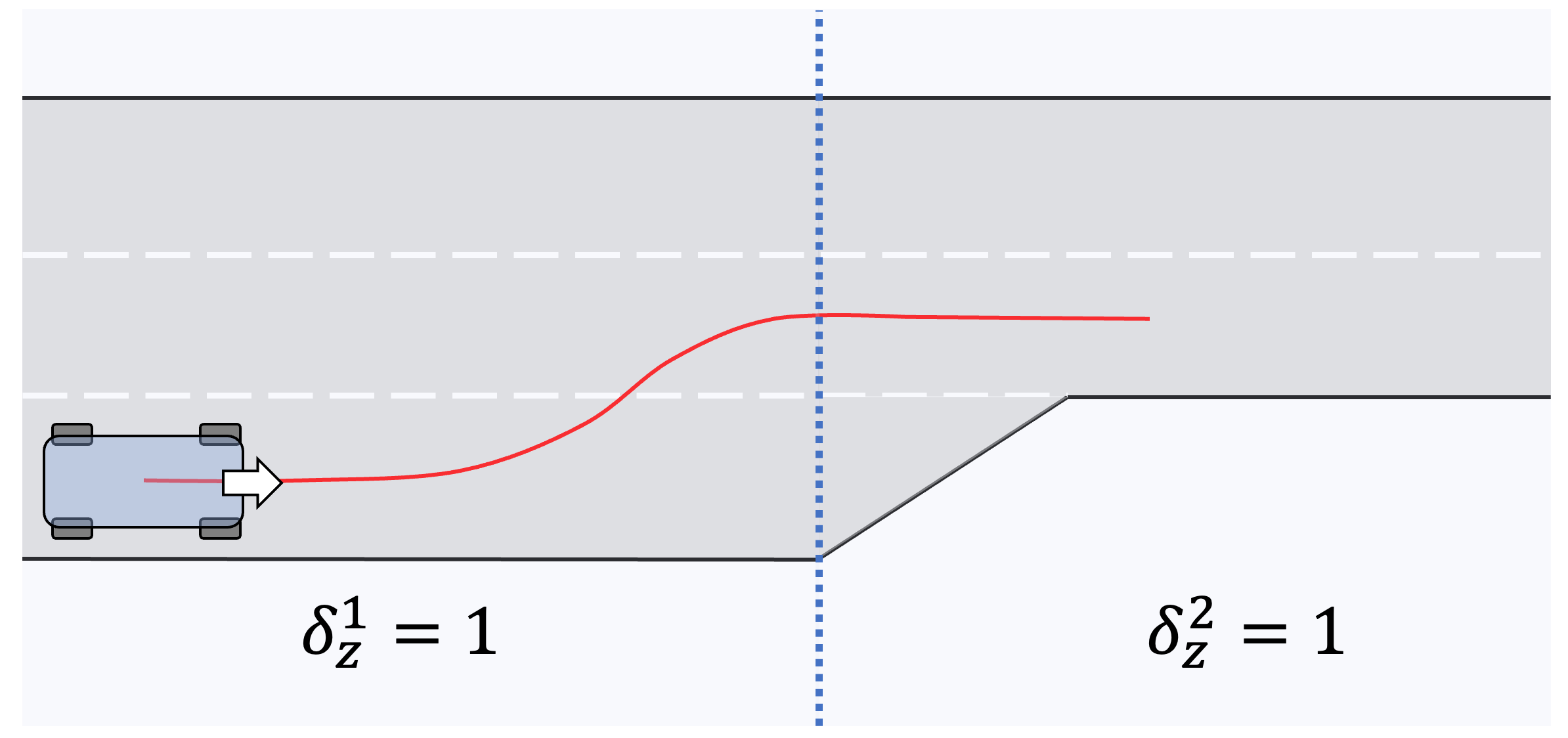}}}		
	\caption{Zone-dependent traffic rule: transition from a zone with three~lanes~($\delta_{\mathrm{z}}^1 = 1$) to a zone with two~lanes~($\delta_{\mathrm{z}}^2=1$), using the proposed MIP inequality constraints in~\eqref{eq:zones} and~\eqref{eq:laneBounds}.}
	\label{fig:zones_figure}
\end{figure}

\subsection{Extended Dynamic System with Auxiliary Variables}

\change{For the prediction model~\eqref{OCP:dyn}, the}
vehicle kinematics~\eqref{eq:kinematics} and the auxiliary dynamics~\eqref{eq:aux_state} result in \change{the augmented} system
\begin{equation}
\begin{aligned}
\begin{bsmallmatrix}
\ps(i+1) \\ \pn(i+1) \\ \vs(i+1) \\ \pref(i+1) \\ n_{\mathrm{LC}}(i+1) 
\end{bsmallmatrix} &= 
\begin{bsmallmatrix}
1 & 0 & \Ts & 0 & 0 \\
0 & 1 & 0 & 0 & 0 \\
0 & 0 & 1 & 0 & 0 \\
0 & 0 & 0 & 1 & 0  \\
0 & 0 & 0 & 0 & 1   \\
\end{bsmallmatrix}
\begin{bsmallmatrix}
\ps(i) \\ \pn(i) \\ \vs(i) \\ \pref(i) \\ n_{\mathrm{LC}}(i) 
\end{bsmallmatrix}
\!+\!
\begin{bsmallmatrix}
0 & 0 & 0 & 0 & 0 \\
0 & \Ts & 0 & 0 & 0 \\
\Ts & 0 & 0 & 0 & 0 \\
0 & 0 & 1 & 0 & 0 \\
0 & 0 & 0 & 1 & 1 \\
\end{bsmallmatrix}
\begin{bsmallmatrix}
\as(i) \\ \vn(i) \\ \pc(i) \\ \lu(i) \\ \ld(i)
\end{bsmallmatrix}.
\end{aligned} \label{eq:sysDyn}
\end{equation}
The MIP-DM \change{also enforces} simple bounds on state variables at each time step $i \in \Z_{0}^{N}$
\begin{equation}
\begin{aligned}
-\frac{\wl}{2} \leq \pn(i) \leq (\nl-\frac{1}{2})\,\wl, \quad \ubar{v}_{\mathrm{s}}(i) \leq \vs(i) \leq \bar{v}_{\mathrm{s}}(i), \\
0 \leq \pref(i) \leq (\nl-1)\,\wl, \quad 0 \leq n_{\mathrm{LC}}(i) \leq n_{\mathrm{LC}}^{\mathrm{max}},
\end{aligned}
\label{eq:bounds_state}
\end{equation}
and simple bounds on control inputs for $i \in \Z_{0}^{N-1}$
\begin{equation}
\ubar{a}_{\mathrm{s}}(i) \leq \as(i) \leq \bar{a}_{\mathrm{s}}(i), \quad \ubar{v}_{\mathrm{n}}(i) \leq \vn(i) \leq \bar{v}_{\mathrm{n}}(i).
\label{eq:bounds_control}
\end{equation}

\subsection{Objective for Decision Making and Motion Planning}
\label{sec:obj_MIQP}

The objective function~\eqref{OCP:obj} of the proposed MIP-DM is
$\sum_{i=0}^{N} \ell_i(x(i),u(i))$, where the stage cost is
\begin{equation}
	\begin{aligned}
		\ell_i  &= w_1\, \Vert \ps(i) - \bpsref(i)  \Vert_2^2 + w_2\, \Vert \pn(i) - \pref(i)  \Vert_2^2 \\
		& + w_3 \, \as(i)^2 + w_4 \, \vn(i)^2 + w_5 \, \lc(i) \\
		& + w_6 \, | \pref(i) - \bpref(i) | + w_7 \, \scol(i),
		\label{eq:stageCost}
	\end{aligned}
\end{equation}
where $\lc(i) = \lu(i)+\ld(i)$, $\scol(i) = \scolx(i) + \scoly(i)$, and $w_j \ge 0$ for $j=1,\ldots,7$ are the weights. The first term in~\eqref{eq:stageCost} is the longitudinal tracking error with respect to a reference trajectory $\bpsref(i)$, e.g., computed based on a desired reference velocity. The second term minimizes the lateral tracking error with respect to the current center lane. The third and fourth terms penalize the control actions, i.e., the longitudinal acceleration and lateral velocities, respectively. The fifth term penalizes lane change decisions. 

The sixth term in~\eqref{eq:stageCost} minimizes a tracking error of the current lane with respect to a given preferred lane value $\bpref(i)$, e.g., the right lane in right-hand traffic or the left \change{most} lane when a vehicle desires to make a left turn at a next traffic intersection. To handle the absolute value in~\eqref{eq:stageCost}, we minimize an auxiliary control variable $\Delta \pref$, satisfying
\begin{equation}
	\Delta \pref \ge \pref - \bpref, \quad \Delta \pref \ge \bpref - \pref,
\end{equation}
such that $\Delta \pref \ge | \pref - \bpref |$ holds. The squared terms in~\eqref{eq:stageCost} may be replaced by absolute values which results in a \change{mixed-integer linear program~(MILP)} instead of an MIQP.
The last term in~\eqref{eq:stageCost} corresponds to a penalty on the slack variables for soft constraint violations. The weight $w_7 \gg 0$ is chosen large enough to ensure that a feasible solution with $\scol(i) = 0$ is found if and when it exists.

The complete MIOCP of the proposed MIP-DM reads as
\begin{equation}
	\begin{aligned}
	\underset{X,\,U}{\text{min}} \quad & \sum_{i=0}^{N} \ell_i(x(i),u(i)) \text{ in Eq.~\eqref{eq:stageCost}} \\
	\text{s.t.} \quad\; & x(0) = \change{\hat{x}_t},  	\\
	& \text{Extended state dynamics in Eq.~\eqref{eq:sysDyn}},  \\
	& \text{Simple bound constraints in Eqs.~\eqref{eq:bounds_state}-\eqref{eq:bounds_control}},  \\
	& \text{Lateral velocity constraint in Eq.~\eqref{eq:steering_new}},  \\
	& \text{Lateral position constraint in Eq.~\eqref{eq:lane_bound}},  \\
	& \text{Lane change constraints in Eq.~\eqref{eq:lane_change}},  \\
	& \text{Time delay constraints in Eqs.~\eqref{eq:timer}-\eqref{eq:min_time}},  \\
	& \text{Obstacle avoidance constraints:} \; \text{Section~\ref{sec:obstacles}},   \\
	& \text{Zone-dependent traffic rules:} \; \text{Section~\ref{sec:zones}}. 
	\end{aligned} \label{eq:MIOCP_complete}
\end{equation}
The state vector is $x = [\ps, \pn, \vs, \pref, n_{\mathrm{LC}}, t_{\mathrm{c}}]$, and the control \change{and auxiliary} input vector is $u = [\as, \vn, \tilde{t}_{\mathrm{c}}, \pc, \deltac, \delta_{\mathrm{o}}, \delta_{\mathrm{z}}]$. The binary optimization variables include the lane change variables $\deltac = [\lu, \ld]$, the obstacle avoidance variables $\delta_{\mathrm{o}} = [\delta_{\mathrm{o}}^1, \ldots, \delta_{\mathrm{o}}^{\nobs}]$, and the traffic zone variables $\delta_{\mathrm{z}} = [\delta_{\mathrm{z}}^1, \ldots, \delta_{\mathrm{z}}^{\nz}]$, while the remaining variables are continuous.

\begin{Remark}
By defining an upper bound on the number of other vehicles for obstacle avoidance in a realistic traffic environment, \change{the MIOCP has} fixed dimensions that allows for static memory allocation in an embedded implementation of the MIP-DM for microprocessors suitable to automotive applications, \change{as discussed later}. \hfill\QEDopen
\end{Remark}

\section{Embedded MIQP Solver for Mixed-Integer Model Predictive Control}
\label{sec:MIMPC_solver}

\change{The MIOCP~\eqref{eq:MIOCP_complete} is converted into the MIQP
\begin{subequations} \label{eq:MIQP}
	\begin{alignat}{5}
		\underset{\z}{\text{min}} \quad & \frac{1}{2} {\z}^{\top} H\, \z + h^\top \z \label{eq:MIQP-primal} \\
		\text{s.t.} \quad\; & G\, \z \;\le \;g, \quad F\, \z \;&&= \;f, \label{eq:MIQP-GF} \\
		\quad\; & \z_{j} \in \Z, \quad && j \in \I, \label{MIQP:int}
	\end{alignat}
\end{subequations} 
where $\z$ includes all optimization variables and the index set $\I$ denotes the integer variables.}
Next, we summarize the main ingredients of the \software{BB-ASIPM} solver~\cite{Quirynen2022} that uses a B\&B method with reliability branching and warm starting~\cite{Hespanhol2019}, block-sparse presolve techniques~\cite{Quirynen2022}, early termination and infeasibility detection~\cite{Liang2021} within a fast convex QP solver~\cite{ASIPM}.

\subsection{Branch-and-bound Method and Search Heuristics}
\label{sec:BandB}

The B\&B algorithm sequentially creates partitions of the original MIQP problem as shown in Figure~\ref{fig:BB_ilu}. For each partition, a local lower bound on the optimal objective value is obtained by solving a convex relaxation of the MIQP subproblem. 
If the relaxation yields an integer-feasible solution, the B\&B updates the global upper bound for the MIQP solution, which is used to \emph{prune} tree partitions. 
The B\&B method terminates when the difference between the upper and lower bound
 is below a user-defined threshold. A key decision of the B\&B procedure is how to create partitions, i.e., which node to choose and which discrete variable to select for branching. \software{BB-ASIPM} uses \emph{reliability branching} which combines strong branching and pseudo-costs~\cite{achterberg2005branching}.

\begin{figure} 
	\centering
	\includegraphics[trim={0 0 0 0},width=0.4\textwidth]{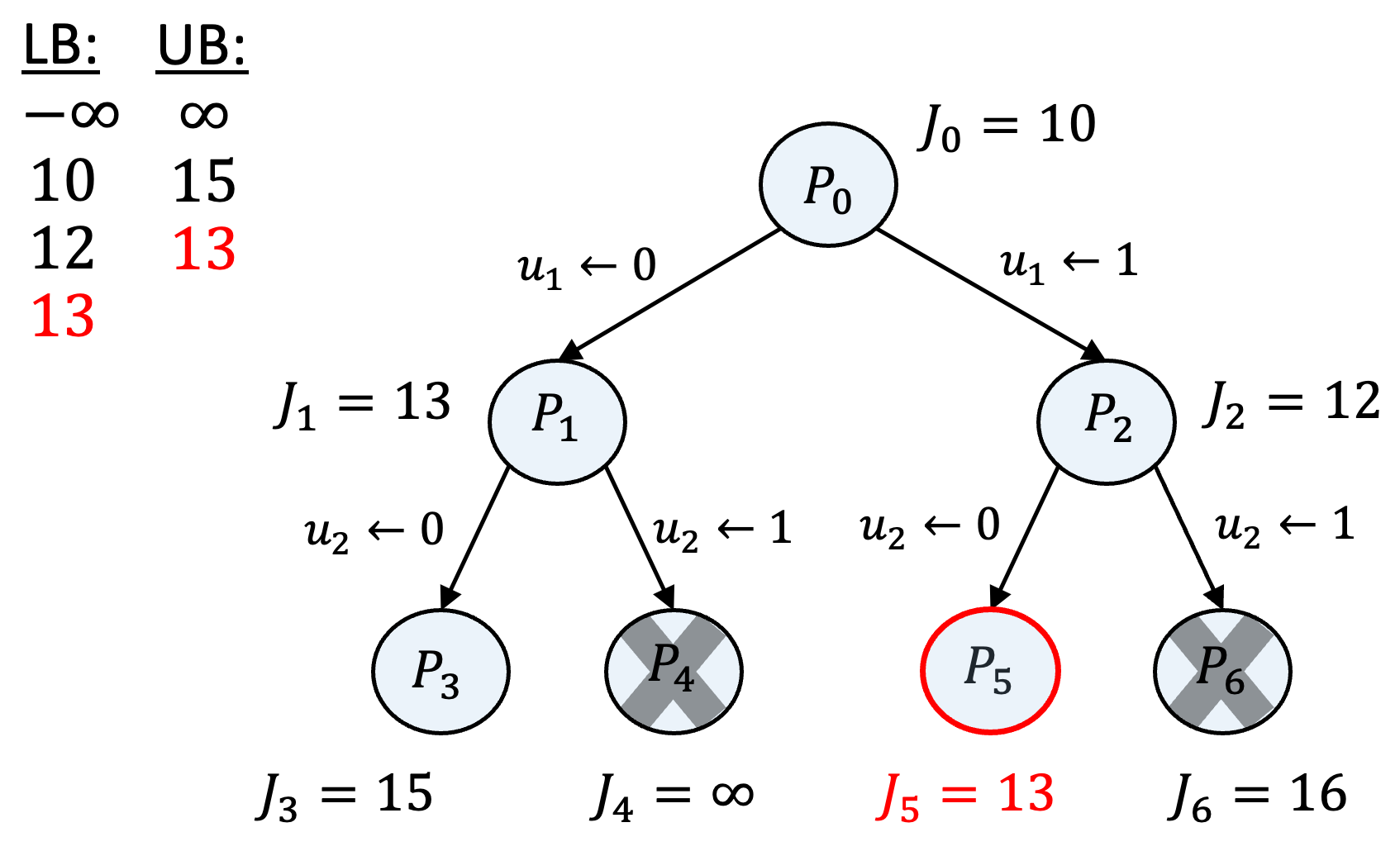}
	\caption{Branch-and-bound~(B\&B) method as a binary search tree. A selected node can be either \emph{branched}, resulting in $2$~partitions for each binary variable $\change{u_j} \in \{0, 1\}$, or \emph{pruned} based on feasibility or the current upper bound.}
	\label{fig:BB_ilu}
\end{figure}

\subsection{Tailored Exact Presolve Reduction Techniques}
\label{sec:presolve}

We refer to the parametric MIQP from~\eqref{eq:MIQP} as $\P(\theta)$, in which the parameter vector $\theta$ includes the state estimate $\change{\hat{x}_t}$, and we denote the discrete variables in~\eqref{MIQP:int} by $\bin \in \Z^{N_\bin}$. We use the compact notation $\P(\theta,\bin_{\fixed} = \hat{\bin})$
to denote the MIQP after fixing $\bin_j = \hat{\bin}_j, j\in\fixed$ where $\fixed$ is an index set.

\begin{Definition}[Presolve Step]
	Given problem $\P(\theta)$ and a set of integer values $\{\hat{\bin}_j\}_{j \in \fixed}$ for the index set $\fixed \subseteq \{ 1, \dots, N_\bin \}$, the presolve step computes
	\begin{equation}
		\{\flag, \hat{\bin}^+,\fixed^+\} \gets \text{Presolve}(\P(\theta),\hat{\bin},\fixed),
	\end{equation}
	resulting in updated integer values $\{\hat{\bin}_j^+\}_{j \in \fixed^+}$ for the index set $\fixed^+ \subseteq \{ 1, \dots, N_\bin \}$, such that:
	\begin{enumerate}
		\itemsep0pt
		\item \change{The} new index set includes the original set, $\fixed \subseteq \fixed^+$.
		\item $\P(\theta,\bin_{\fixed^+} = \hat{\bin}^+)$ is infeasible / unbounded only if $\P(\theta,\bin_{\fixed} = \hat{\bin})$ is infeasible / unbounded.
		\item Any feasible / optimal solution of $\P(\theta,\bin_{\fixed^+} = \hat{\bin}^+)$ maps to a feasible / optimal solution of $\P(\theta,\bin_{\fixed} = \hat{\bin})$, with identical objective value. \hfill\QEDopen
	\end{enumerate}
	\label{def:presolve}
\end{Definition}
A presolve routine applied to a root node in B\&B corresponds to Definition~\ref{def:presolve} with $\fixed=\emptyset$.
In general, presolve cannot prune all of the binary or integer decision variables, but often it leads to a reduced problem that is significantly faster to solve. 

We use the tailored block-sparse presolve procedure~\cite[Section~4]{Quirynen2022} that abides by the rules in Def.~\ref{def:presolve}, \change{and includes:}
\begin{itemize}
	\item \emph{Domain propagation} to strengthen bounds based on constraints of the MIQP, which may lead to fixing multiple integer variables. A tailored implementation for MIOCPs based on an iterative forward-backward propagation is described in~\cite[Alg.~2]{Quirynen2022}.
	\item \emph{Redundant constraints} are detected and removed based on updated bound values, which may also benefit \emph{dual fixing} of multiple variables, see~\cite[Alg.~4]{Quirynen2022}.
	\item \emph{Coefficient strengthening} to tighten the feasible space of the convex QP relaxation without removing any integer-feasible solution of the MIQP. A block-sparse implementation is described in~\cite[Alg.~5]{Quirynen2022}.
	\item \emph{Variable probing} to obtain tightened bound values for multiple optimization variables by temporarily fixing a binary variable to~$0$ and~$1$, see~\cite[Alg.~6]{Quirynen2022}.
\end{itemize}
The presolve procedure in~\cite{Quirynen2022} terminates if the problem is detected to be infeasible or if insufficient progress is made from one iteration to the next.
An upper limit on the number of presolve iterations and/or a timeout is typically needed to ensure computational efficiency, and it generally results in a considerable speedup of the B\&B computations.

\subsection{Block-sparse QP solver for Convex Relaxations}
\label{sec:convexQP}

A primal-dual interior point method~(IPM) uses a Newton-type algorithm to solve a sequence of relaxed Karush-Kuhn-Tucker~(KKT) conditions for the convex QP.
We use the active-set based inexact Newton implementation of \software{ASIPM}~\cite{ASIPM}, \change{which exploits the} block-sparse structure in the linear system, \change{with improved numerical conditioning,} reduced matrix factorization updates, warm starting, early termination and infeasibility detection~\cite{Liang2021}.
\change{If the convex QP relaxation}
\begin{itemize}
	\item is infeasible,
	\item has \change{optimal} value that exceeds the current global upper bound in the B\&B method,
\end{itemize}
the node and corresponding subtree can be pruned from the B\&B tree. A considerable computational effort can be avoided if the above scenarios are detected early, i.e., more quickly than solving the convex QPs. In~\cite{Liang2021}, \change{we} describe an early termination method based on a tailored dual feasibility projection strategy \change{applicable to \software{BB-ASIPM}} to handle both cases and to reduce the computational effort of the B\&B method without affecting the quality of the optimal solution.

\subsection{Embedded Software Implementation for \change{Hybrid MPC}}
\label{sec:software}

In \change{hybrid MPC}, warm starting can be used to reduce the computational effort in the B\&B method from one time step to the next as discussed in~\cite{Bemporad2018,Marcucci2021}. \change{\software{BB-ASIPM} uses} \emph{tree propagation}~\cite{Hespanhol2019,Quirynen2022} to efficiently reuse the branching decisions and pseudo-costs from the previous MIQP solution. An upper bound can be imposed on the number of B\&B iterations to ensure a maximum computation time below a threshold. If an integer-feasible solution is found, a B\&B method automatically provides a bound on the suboptimality of this MIQP solution. The \software{BB-ASIPM} solver is implemented in self-contained C~code, which allows for real-time implementations on embedded \change{microprocessors} as shown next.

\section{Numerical Simulation Results}
\label{sec:simulation}

We present numerical simulation results for the MIP-DM described in Section~\ref{sec:MIQP_form}, in a variety of traffic scenarios. We also compare the \software{BB-ASIPM} solver \change{from Section~\ref{sec:MIMPC_solver}} against state-of-the-art software tools, and we demonstrate its real-time feasibility on dSPACE rapid prototyping units.

\subsection{Problem Formulation and Simulation Test Scenarios}

\begin{table}[t]
	\normalsize
		\caption{Problem dimensions and parameters in MIQP formulation of Section~\ref{sec:MIQP_form} for each of the test scenarios in Fig.~\ref{fig:test_scenarios}. The number of binary variables per time step in the MIOCP prediction time horizon is $n_\delta = 2+3\,\nobs+\nz$.}
	\label{tab:test_scenarios}
	\centering
	\setlength{\tabcolsep}{0.6em}
	\begin{tabular}{ l | c c c c c c c }
		\toprule
		& $N$ & $\nx$ & $\nU$ & $n_\delta$ & $\nc$ & $\nobs$ & $\nz$  \\
		\midrule
		\case{1} see Fig.~\ref{fig:test_scenario1} & 15 & 6 & 20 & 14 & 60 & 3 & 3  \\ 
		\case{2} see Fig.~\ref{fig:test_scenario2} & 15 & 6 & 18 & 12 & 56 & 3 & 1  \\ 
		\case{3} see Fig.~\ref{fig:test_scenario3} & 15 & 6 & 23 & 17 & 71 & 4 & 3  \\ 
		\case{4} see Fig.~\ref{fig:test_scenario4} & 15 & 6 & 24 & 18 & 73 & 4 & 4  \\ 
		\case{5} see Fig.~\ref{fig:test_scenario5} & 15 & 6 & 16 & 10 & 47 & 2 & 2  \\ 
		\case{6} see Fig.~\ref{fig:test_scenario6} & 15 & 6 & 17 & 11 & 49 & 2 & 3  \\ 
		\case{7} see Fig.~\ref{fig:test_scenario7} & 15 & 6 & 20 & 14 & 60 & 3 & 3  \\ 
		\case{8} see Fig.~\ref{fig:test_scenario8} & 15 & 6 & 20 & 14 & 60 & 3 & 3  \\ 
		\bottomrule
	\end{tabular}
\end{table}

\begin{figure*}
	\centering
	\hspace{-1cm}
	\begin{minipage}{0.4\linewidth}%
		\begin{subfigure}[b]{\textwidth}
			\centering
			\includegraphics[width=1.0\textwidth,trim={0cm 0cm 0 0cm},clip]{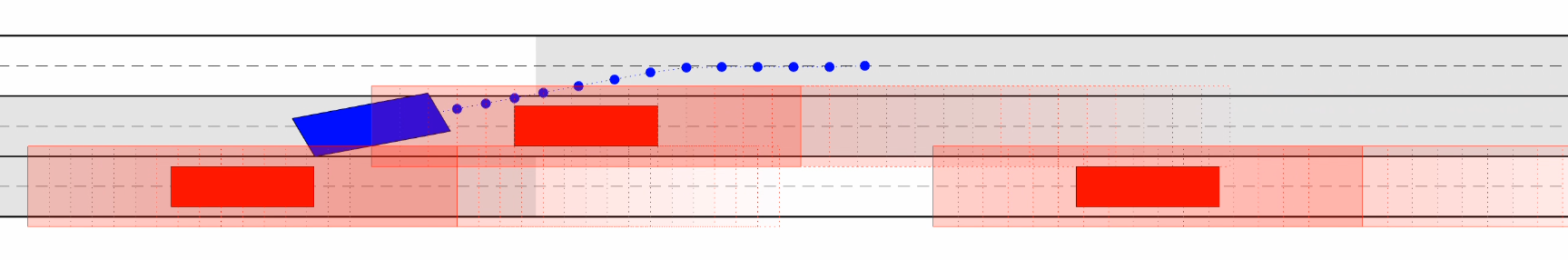}
			\captionsetup{font=normalsize,labelfont={sf}}
			\caption[]%
			{\normalsize Scenario~1: ego vehicle overtaking three obstacles 
				on a road with one-way traffic.}
			\label{fig:test_scenario1}
		\end{subfigure}
		\vskip\baselineskip\vspace{2mm}
		\begin{subfigure}[b]{\textwidth}
			\centering
			\includegraphics[width=1.0\textwidth,trim={0cm 0cm 0 0cm},clip]{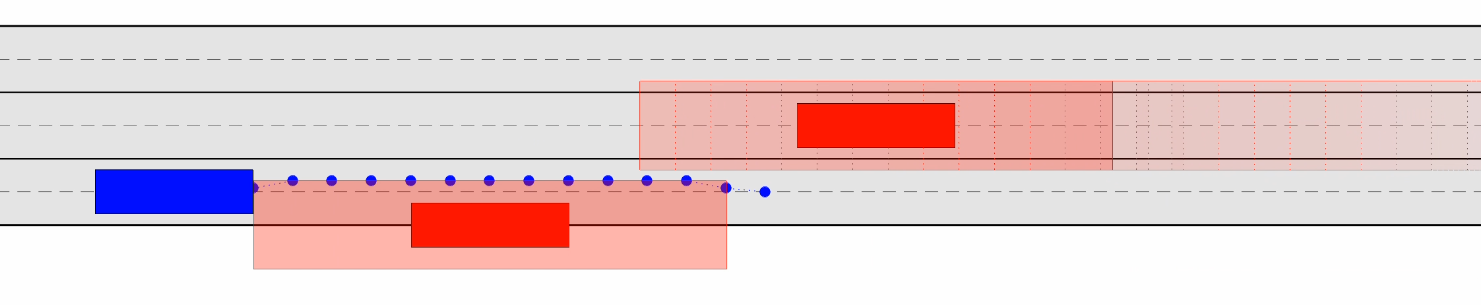}
			\captionsetup{font=normalsize,labelfont={sf}}
			\caption[]%
			{\normalsize Scenario~2: ego vehicle swaying for two parked vehicles (only one visible), avoiding a third vehicle on other lane with one-way traffic.}
			\label{fig:test_scenario2}
		\end{subfigure}
		\vskip\baselineskip\vspace{2mm}
		\begin{subfigure}[b]{\textwidth}
			\centering
			\includegraphics[width=1.0\textwidth,trim={0cm 0cm 0 0cm},clip]{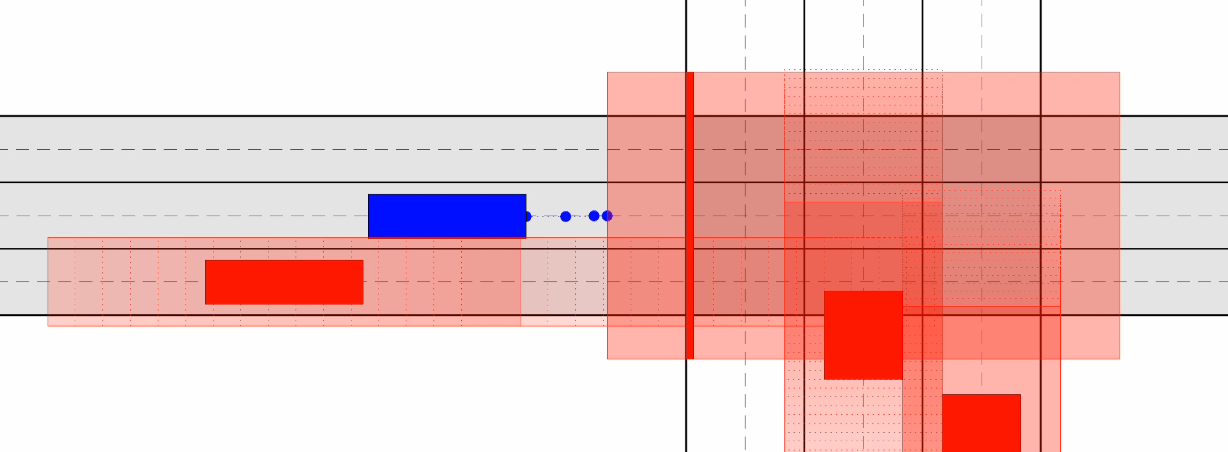}
			\captionsetup{font=normalsize,labelfont={sf}}
			\caption[]%
			{\normalsize Scenario~3: ego vehicle overtaking before stopping at intersection, then ego continues after two~vehicles finish crossing intersection.}
			\label{fig:test_scenario3}
		\end{subfigure}
		\vskip\baselineskip\vspace{2mm}
		\begin{subfigure}[b]{\textwidth}
			\centering
			\includegraphics[width=1.0\textwidth,trim={0cm 0cm 0 0cm},clip]{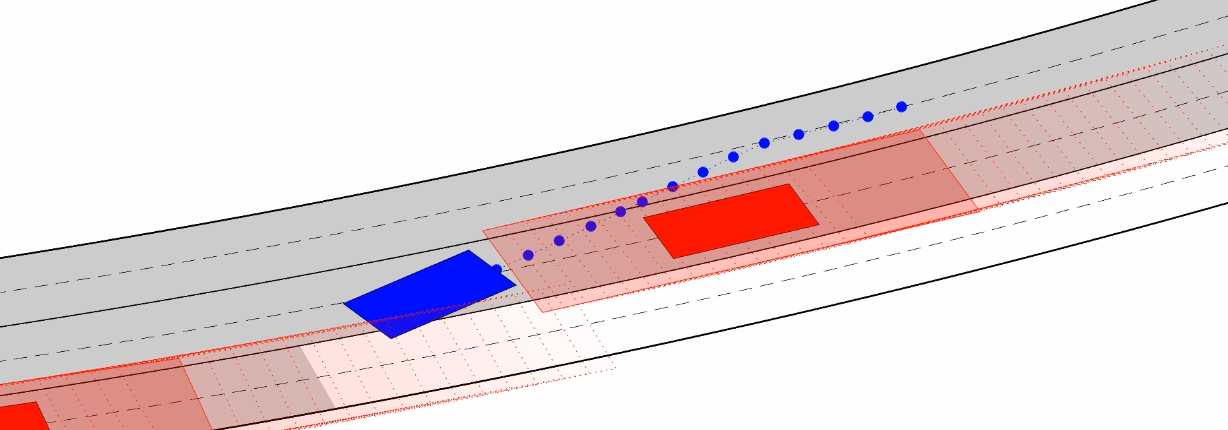}
			\captionsetup{font=normalsize,labelfont={sf}}
			\caption[]%
			{\normalsize Scenario~4: ego vehicle overtaking obstacles on a curved road with one-way traffic, followed by stopping and crossing an intersection.}
			\label{fig:test_scenario4}
		\end{subfigure}
	\end{minipage}
	\begin{minipage}{0.1\linewidth}%
		\hspace{5mm}
	\end{minipage}
	\begin{minipage}{0.4\linewidth}%
		\begin{subfigure}[b]{\textwidth}
			\centering
			\includegraphics[width=1.1\textwidth,trim={0cm 0cm 0 0cm},clip]{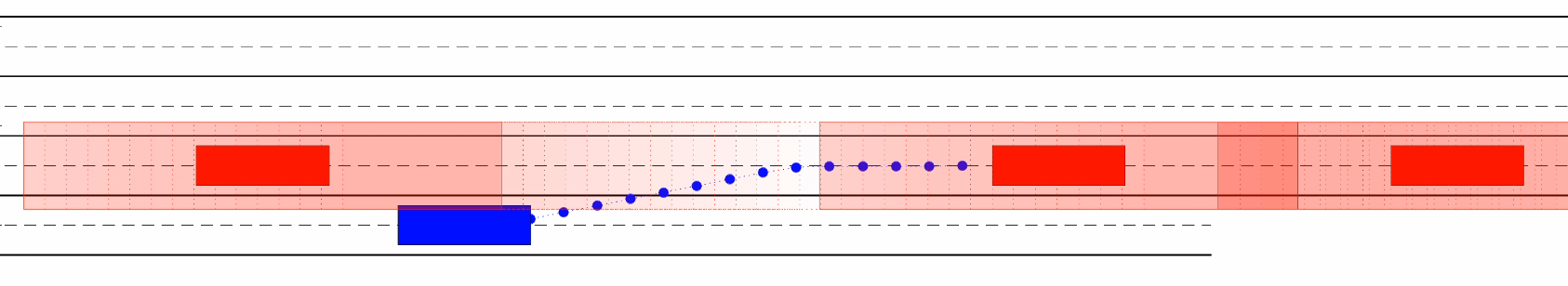}
			\captionsetup{font=normalsize,labelfont={sf}}
			\caption[]%
			{\normalsize Scenario~5: ego vehicle merging to lane~2 between three~vehicles with one-way traffic.}
			\label{fig:test_scenario5}
		\end{subfigure}
		\vskip\baselineskip\vspace{2mm}
		\begin{subfigure}[b]{\textwidth}
			\centering
			\includegraphics[width=1.1\textwidth,trim={0cm 0cm 0 0cm},clip]{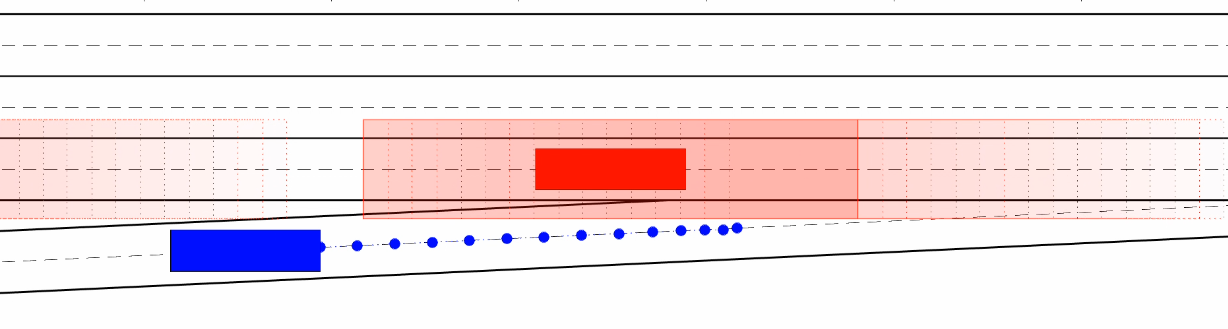}
			\captionsetup{font=normalsize,labelfont={sf}}
			\caption[]%
			{\normalsize Scenario~6: ego vehicle merging at the end of lane onto a new lane while avoiding / overtaking three~vehicles (only one visible).}
			\label{fig:test_scenario6}
		\end{subfigure}
		\vskip\baselineskip\vspace{2mm}
		\begin{subfigure}[b]{\textwidth}
			\centering
			\includegraphics[width=1.1\textwidth,trim={0cm 0cm 0 0cm},clip]{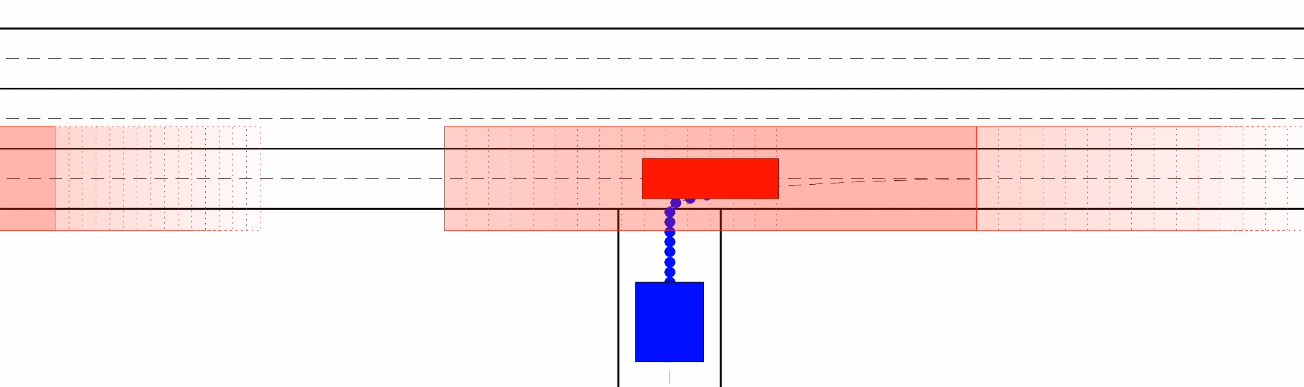}
			\captionsetup{font=normalsize,labelfont={sf}}
			\caption[]%
			{\normalsize Scenario~7: ego vehicle performs right turn at a T-intersection, merging between two vehicles (only one visible) on same lane of the road segment.}
			\label{fig:test_scenario7}
		\end{subfigure}
		\vskip\baselineskip\vspace{2mm}
		\begin{subfigure}[b]{\textwidth}
			\centering
			\includegraphics[width=1.1\textwidth,trim={0cm 0cm 0 0cm},clip]{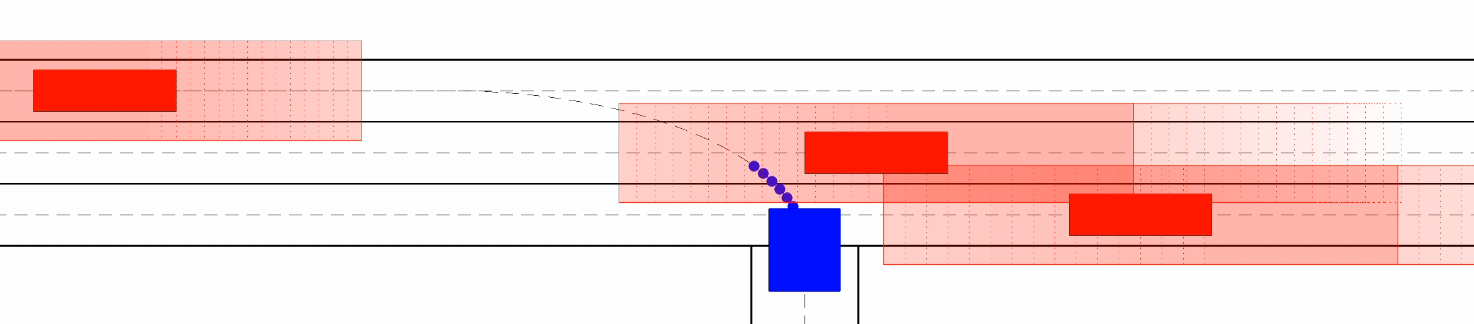}
			\captionsetup{font=normalsize,labelfont={sf}}
			\caption[]%
			{\normalsize Scenario~8: ego vehicle turns left at T-intersection, following one vehicle while avoiding two other vehicles driving in the opposite direction.}
			\label{fig:test_scenario8}
		\end{subfigure}
	\end{minipage}
	\caption[ ]{Snapshot of the closed-loop Matlab simulations using the MIP-DM in $8$~test scenarios. \change{The ego vehicle is shown in blue, other vehicles in red.} A video recording of the simulations is available at: \url{https://youtu.be/FyaGRZvuqmA}.}
	\label{fig:test_scenarios}
\end{figure*}

In this section, we perform closed-loop simulations of MIP-DM in Matlab using the vehicle model in~\eqref{eq:kinematics},
\change{to} show the variety of traffic scenarios that can be handled explicitly using the MIQP in Section~\ref{sec:MIQP_form}. \change{We use a simple model~\eqref{eq:kinematics} to assess the behavior and the stand-alone computational load of MIP-DM. Robustness to model approximations and uncertainty is validated in the experiments shown later.}

Figure~\ref{fig:test_scenarios} shows a snapshot of the Matlab simulations for $8$~test scenarios.
Table~\ref{tab:test_scenarios} shows the problem dimensions and parameter values in the MIQP formulation of Section~\ref{sec:MIQP_form} for the test scenarios in Fig.~\ref{fig:test_scenarios},
where $N=15$ is the horizon length, $\nx$ is the number of state variables, $\nU$ is the number of control variables, $\nc$ is the number of inequality constraints, each per time step, and $n_\delta = 2+3\,\nobs+\nz$ is the number of binary variables per time step, \change{with} $\nobs$ the maximum number of obstacles~(see Section~\ref{sec:obstacles}), and $\nz$ the number of zones~(see Section~\ref{sec:zones}). Using a sampling time of $\Ts = 1$~s, the MIP-DM time horizon is $T = N\, \Ts = 15$~s.

Scenario~1 in Fig.~\ref{fig:test_scenario1} shows the ego vehicle overtaking three obstacles, where two obstacles are on lane~1 and a third obstacle is on lane~2, on a road segment with one-way traffic. \change{Lane~1 refers to the right most lane with respect to the ego vehicle's direction of motion.} Scenario~2 in Fig.~\ref{fig:test_scenario2} shows the ego vehicle swaying around two parked vehicles~(with zero velocity) on lane~1, while avoiding a third vehicle on lane~2.
Scenario~3 in Fig.~\ref{fig:test_scenario3} shows the ego vehicle overtaking one vehicle on lane~1 before stopping at a traffic intersection, then crossing after two other vehicles. Scenario~4 in Fig.~\ref{fig:test_scenario4} shows the ego vehicle overtaking three obstacles (two~vehicles on lane~1 and one~vehicle on lane~2) on a curved road segment with one-way traffic, followed by stopping and crossing an intersection. In the test scenarios~1-4, lane~1 is the preferred lane $\bpref$ in~\eqref{eq:stageCost}, so that the ego vehicle always returns to lane~1 after each overtaking or sway maneuver. 

Scenario~5 in Fig.~\ref{fig:test_scenario5} shows the ego vehicle merging from lane~1 to lane~2 between three vehicles on lane~2, i.e., the preferred lane $\bpref$ in~\eqref{eq:stageCost} is lane~2. Scenario~6 in Fig.~\ref{fig:test_scenario6} shows the ego vehicle merging at the end of a current lane onto a new lane while avoiding and/or overtaking three vehicles that are driving on the same lane. Scenario~7 in Fig.~\ref{fig:test_scenario7} shows the ego vehicle performing a right turn at a T-intersection, merging between two vehicles on the same lane of the new road segment. Scenario~8 in Fig.~\ref{fig:test_scenario8} shows the ego vehicle performing a left turn at a T-intersection, following one vehicle on the same lane while avoiding two other vehicles driving in the opposite direction.
In the test scenarios~5-8, after a merging or turning maneuver, the ego vehicle overtakes any other vehicle that is driving below the speed limit.

\subsection{Computational Performance and Solver Comparisons}

\begin{table*}[t]
	\normalsize
	\parbox{\textwidth}{
		\caption{Average and worst-case computation times for each of the $8$~scenarios in Figure~\ref{fig:test_scenarios} for MIP-DM with the MIQP formulation in Section~\ref{sec:MIQP_form}, \change{using} \software{GUROBI}, \software{MOSEK} and \software{BB-ASIPM} solver.}
		\label{tab:Matlab_results}
	}
	\centering
	\setlength{\tabcolsep}{1em}
	\begin{tabular}{ l | c c | c c | c c }
		\toprule
		& \multicolumn{2}{c}{\software{GUROBI}} & \multicolumn{2}{c}{\software{MOSEK}} & \multicolumn{2}{c}{\software{BB-ASIPM}}  \\
		& Mean time & Max time & Mean time & Max time & Mean time & Max time  \\
		\midrule
		\case{1} see Fig.~\ref{fig:test_scenario1} & $9.2$~ms  &  $21.8$~ms & $116.2$~ms  &  $466.6$~ms & $16.3$~ms  &  $65.9$~ms  \\ 
		\case{2} see Fig.~\ref{fig:test_scenario2} & $4.1$~ms  &  $11.7$~ms & $25.2$~ms  &  $140.5$~ms & $6.0$~ms  &  $37.7$~ms  \\ 
		\case{3} see Fig.~\ref{fig:test_scenario3} & $4.9$~ms  &  $15.3$~ms & $47.1$~ms  &  $180.0$~ms & $7.2$~ms  &  $38.5$~ms  \\ 
		\case{4} see Fig.~\ref{fig:test_scenario4} & $4.7$~ms  &  $17.3$~ms & $41.3$~ms  &  $160.9$~ms & $7.3$~ms  &  $48.7$~ms  \\ 
		\case{5} see Fig.~\ref{fig:test_scenario5} & $3.9$~ms  &  $14.3$~ms & $37.0$~ms  &  $231.7$~ms & $6.0$~ms  &  $45.0$~ms  \\ 
		\case{6} see Fig.~\ref{fig:test_scenario6} & $4.4$~ms  &  $16.8$~ms & $44.6$~ms  &  $198.9$~ms & $6.1$~ms  &  $39.3$~ms  \\ 
		\case{7} see Fig.~\ref{fig:test_scenario7} & $3.6$~ms  &  $15.7$~ms & $23.4$~ms  &  $235.6$~ms & $4.7$~ms  &  $39.7$~ms  \\ 
		\case{8} see Fig.~\ref{fig:test_scenario8} & $3.2$~ms  &  $15.1$~ms & $21.3$~ms  &  $130.7$~ms & $3.8$~ms  &  $29.4$~ms  \\ 
		\bottomrule
	\end{tabular}
\end{table*}

Table~\ref{tab:Matlab_results} shows the average and worst-case computation times of MIP-DM for each of the $8$~simulation scenarios that are illustrated in Figure~\ref{fig:test_scenarios}, using the MIQP formulation as described in Section~\ref{sec:MIQP_form} and where the MIQPs at each control time step are solved using either \software{GUROBI}, \software{MOSEK} or \software{BB-ASIPM}. It can be observed that the average and worst-case computation times of \software{BB-ASIPM} are approximately~$6$ and $5$ times faster than \software{MOSEK}, respectively. On the other hand, the average and worst-case computation times of \software{GUROBI} are approximately~$1.5$ and $2.5$ times faster than \software{BB-ASIPM}, respectively. \change{Note that all default presolve options are enabled in the \software{GUROBI} solver.}

Given the relatively simple and compact algorithmic implementation in \software{BB-ASIPM}, e.g., compared to the extensive collection of advanced heuristics, presolve and cutting plane techniques in the commercial \software{GUROBI}~\cite{gurobi} solver, it is reassuring to see that the tailored \software{BB-ASIPM} solver can remain competitive with state-of-the-art software tools in Table~\ref{tab:Matlab_results}.
The software implementation of \software{BB-ASIPM}~\cite{Quirynen2022} is relatively compact and self-contained such that it can execute on an embedded microprocessor for real-time vehicle decision making and motion planning. Instead, state-of-the-art optimization tools, such as \software{GUROBI} and \software{MOSEK} typically cannot be used on embedded control hardware with limited computational resources and available memory~\cite{DiCairano2018tutorial}.

\subsection{Hardware-in-the-loop Simulation Results on dSPACE Scalexio and MicroAutoBox-III Rapid Prototyping Units}
\label{sec:dSPACE}

Next, we present detailed results of running hardware-in-the-loop simulations for each of the $8$~test scenarios shown in Figure~\ref{fig:test_scenarios} on both the dSPACE Scalexio\footnote{dSPACE Scalexio DS6001 unit, with an Intel i7-6820EQ quad-core $2.8$~GHz processor with 64 kB L1 cache per core, 256 kB L2 cache per core, 8 MB shared L3 cache, 4 GB DDR4 RAM, and 8 GB flash memory. \change{In the presented results, MIP-DM executes in a single core.}} and the dSPACE MicroAutoBox-III~(MABX-III)\footnote{dSPACE MicroAutoBox-III DS1403 unit, with four ARM Cortex-A15 processor cores with 32 kB L1 cache per core, 4 MB shared L2 cache, 2 GB DDR3L RAM, and 64 MB flash memory. \change{In the presented results, MIP-DM executes in a single core.}} rapid prototyping units.
Table~\ref{tab:dSPACE_results} shows the average and worst-case computation times, the number of B\&B iterations, total number of ASIPM iterations, and the memory usage of the \software{BB-ASIPM} solver on Scalexio and MABX-III.
The memory usage is categorized into \emph{text} that contains code and constant data, which is typically stored in ROM, and \emph{data} that is stored in RAM.

From Table~\ref{tab:dSPACE_results}, MIP-DM is real-time feasible using the proposed \software{BB-ASIPM} solver for each of the $8$~simulation scenarios on both the dSPACE Scalexio and MABX-III units, as the worst-case computation time is below the sampling time of $\Ts = 1$~s at each time step. 
More specifically, considering all test scenarios, the computation times on the dSPACE Scalexio are always below $200$~ms, below $100$~ms $99$\% of the times, and the average is only $17.3$~ms. On MABX-III, the computation times are \change{always} below $800$~ms, below $400$~ms $99$\% of the times, and the average is only $76.3$~ms. 
The total memory usage is approximately $18$~MB on Scalexio and $16.1$~MB on MABX-III, due to the different compilers. 
As expected, for each test scenario, Table~\ref{tab:dSPACE_results} shows that the number of iterations on Scalexio and MABX-III is identical.

\begin{table*}[t]
\normalsize
\parbox{\textwidth}{
	\caption{Average and worst-case computation times, number of B\&B iterations, total number of ASIPM iterations, and memory footprint of the embedded \software{BB-ASIPM} solver on the dSPACE Scalexio and on the dSPACE MABX-III, for hardware-in-the-loop simulations of the MIP-DM method for the $8$~scenarios in Figure~\ref{fig:test_scenarios}.}
	\label{tab:dSPACE_results}
}
\centering
\setlength{\tabcolsep}{0.6em}
\begin{tabular}{ l | c c | c c | c c | c c | c c | c c }
	\toprule
	\multicolumn{1}{l}{} & \multicolumn{4}{c}{\software{BB-ASIPM} solver} & \multicolumn{4}{c}{\software{BB-ASIPM} on {\bf dSPACE Scalexio}} & \multicolumn{4}{c}{\software{BB-ASIPM} on {\bf dSPACE MABX-III}} \\
	\multicolumn{1}{l}{} & \multicolumn{2}{c}{B\&B iters} & \multicolumn{2}{c}{ASIPM iters} & \multicolumn{2}{c}{CPU time~[ms]} & \multicolumn{2}{c}{Memory~[KB]} & \multicolumn{2}{c}{CPU time~[ms]} & \multicolumn{2}{c}{Memory~[KB]} \\
	& mean & max & mean & max & mean & max & text & data & mean & max & text & data \\
	\midrule
	\case{1} see Fig.~\ref{fig:test_scenario1}  &  $6.3$  &  $65$  &  $56.1$  &  $495$ & $22.7$  &  $185.2$ & 170 & 13633 & $99.9$  &  $774.0$ & 132 & 12217 \\ 
	\case{2} see Fig.~\ref{fig:test_scenario2}  &  $4.7$  &  $35$  &  $41.2$  &  $199$ & $14.0$  &  $56.0$ & 167 & 12639 & $63.4$  &  $240.1$  & 132 & 11407 \\ 
	\case{3} see Fig.~\ref{fig:test_scenario3}  &  $4.3$  &  $25$  &  $42.2$  &  $153$ & $16.9$  &  $60.9$ & 170 & 18015 & $74.6$  &  $252.8$ & 133 & 16209  \\ 
	\case{4} see Fig.~\ref{fig:test_scenario4}  &  $7.5$  &  $39$  &  $66.4$  &  $353$ & $27.5$  &  $131.7$ & 177 & 18258 & $119.6$  &  $542.8$ & 134 & 16343  \\ 
	\case{5} see Fig.~\ref{fig:test_scenario5}  &  $4.1$  &  $27$  &  $37.6$  &  $241$ & $13.2$  &  $71.8$ & 170 & 14391 & $59.2$  &  $300.2$ & 132 & 12975  \\ 
	\case{6} see Fig.~\ref{fig:test_scenario6}  &  $7.9$  &  $45$  &  $73.4$  &  $351$ & $18.8$  &  $102.2$ & 170 & 11896 & $83.9$  &  $425.6$ & 133 & 10897  \\ 
	\case{7} see Fig.~\ref{fig:test_scenario7}  &  $3.8$  &  $37$  &  $41.1$  &  $204$ & $12.5$  &  $61.4$ & 170 & 12364 & $54.6$  &  $251.5$ & 133 & 11283  \\ 
	\case{8} see Fig.~\ref{fig:test_scenario8}  &  $4.4$  &  $37$  &  $42.3$  &  $261$ & $12.8$  &  $74.1$ & 170 & 14391 & $55.0$  &  $308.1$ & 132 & 12975  \\
	\bottomrule
\end{tabular}
\end{table*}

\section{Experimental Results of MIP-DM and NMPC on Small-scale Automated Vehicles}
\label{sec:results}

Next, we validate the performance of MIP-DM on experiments with small-scale vehicles, using ROS and an Optitrack motion-capture system~\cite{Berntorp2018}. First, we briefly present the hardware and software setup, then we describe the integration of MIP-DM with a nonlinear MPC~(NMPC) for reference tracking, and finally we show the experiment results.

\subsection{Hardware Setup and Software Implementation}

\begin{figure*}
	\centering
	\begin{minipage}{.3\linewidth}%
		\begin{subfigure}[b]{\textwidth}
			\centering
			\includegraphics[width=0.45\textwidth,trim={0 15cm 0 20cm},clip]{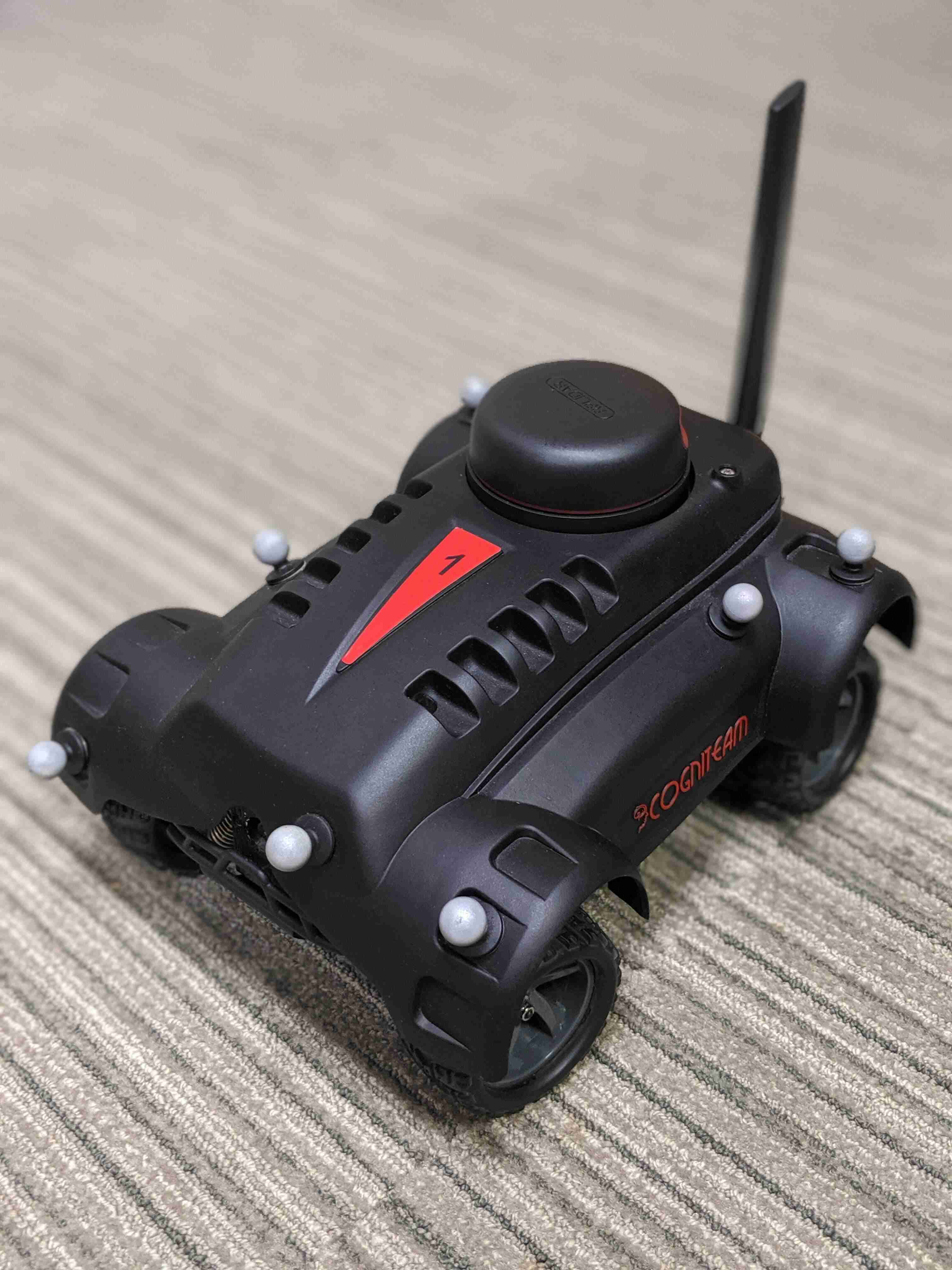}
			\captionsetup{font=normalsize,labelfont={sf}}
			\caption[]%
			{\normalsize Small-scale autonomous vehicle.}    
			\label{fig:hamster}
		\end{subfigure}
		\vskip\baselineskip\vspace{2mm}
		\begin{subfigure}[b]{\textwidth}   
			\centering 
			\includegraphics[width=0.45\textwidth,trim={10cm 22cm 10cm 30cm},clip]{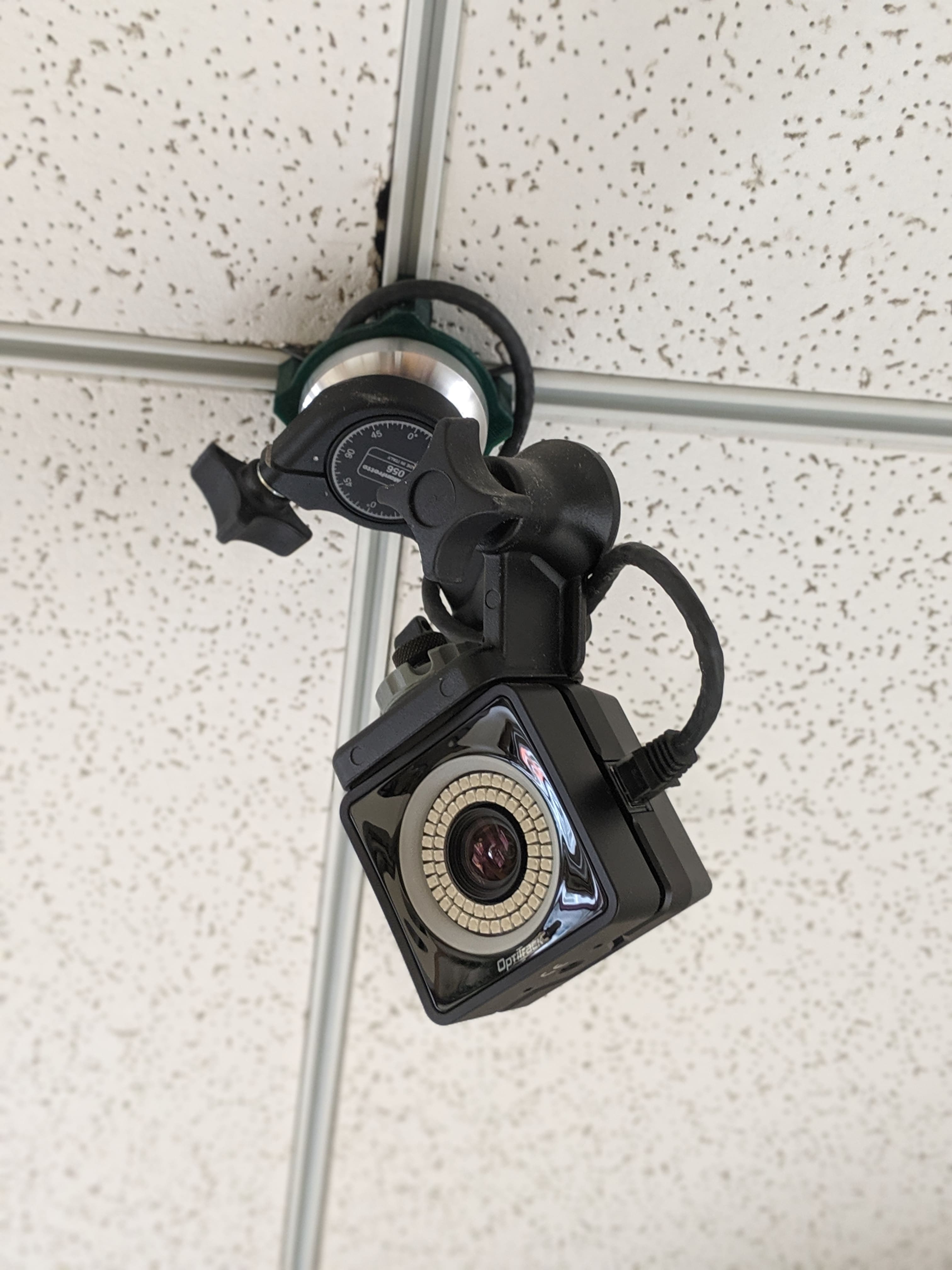}
			\captionsetup{font=normalsize,labelfont={sf}}
			\caption[]%
			{\normalsize OptiTrack motion capture camera.}    
			\label{fig:optitrack}
		\end{subfigure}
	\end{minipage}
\begin{minipage}{.03\linewidth}%
\hspace{1mm}
\end{minipage}
	\begin{minipage}{.6\linewidth}%
		\begin{subfigure}[b]{\textwidth}  
			\centering 
			\includegraphics[trim={0 0 0 0},clip,width=1\textwidth]{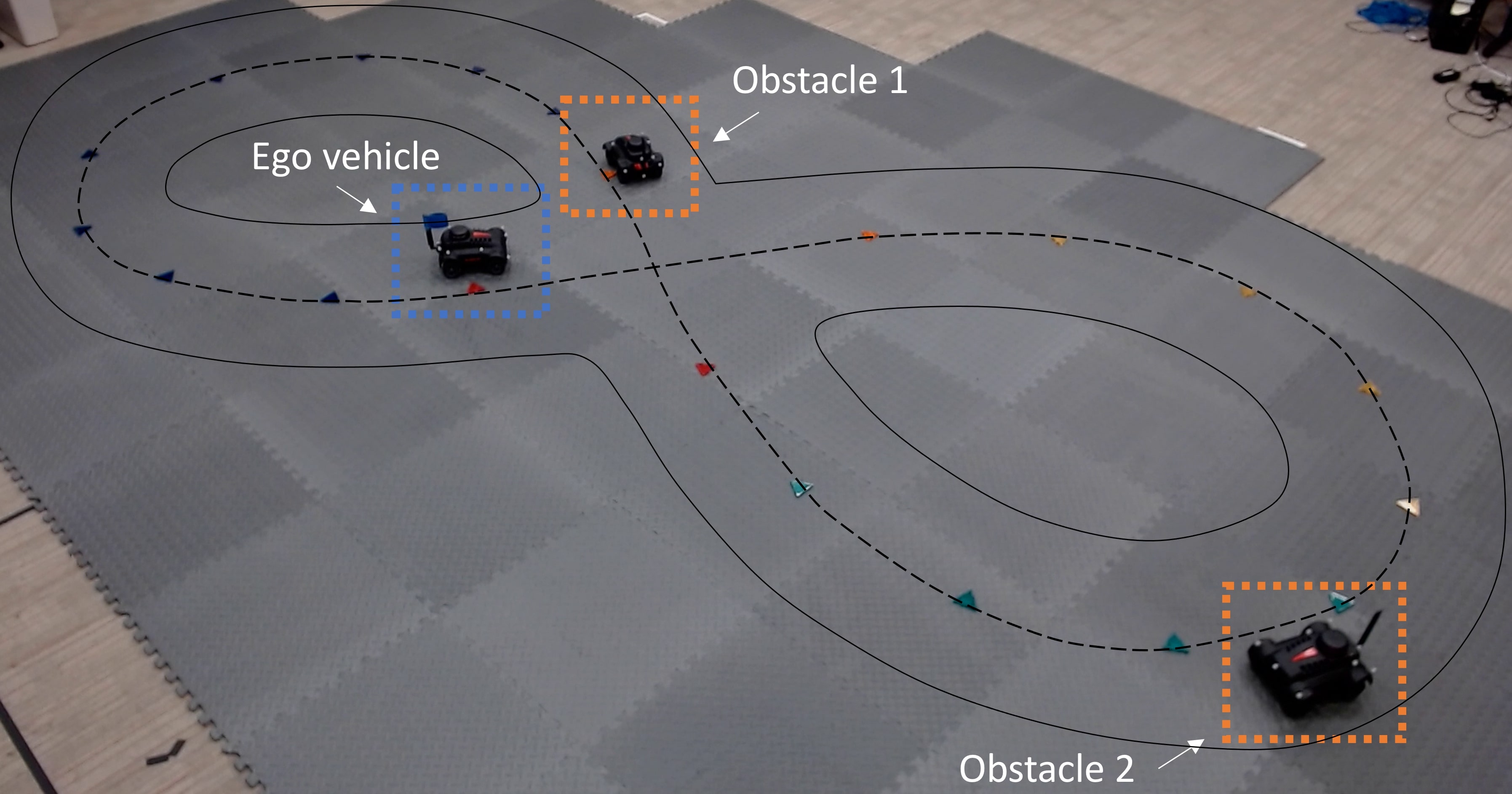}
			\captionsetup{font=normalsize,labelfont={sf}}
			\caption[]%
			{\normalsize Experiments using three small-scale vehicles: the ego vehicle (blue flag) executing the MIP-DM and NMPC, and two obstacles (no flag).}
			\label{fig:track}
		\end{subfigure}
	\end{minipage}
	\caption[ ]{Experimental testbench that consists of small-scale automated vehicles~(a) with on-board sensors, and an OptiTrack motion capture system~(b). Track and snapshot of the positions of the ego vehicle and of the two obstacle vehicles~(c).
	}
	\label{fig:hamster_setup}
\end{figure*}

The hardware setup is illustrated in Figure~\ref{fig:hamster_setup}. It includes a \emph{Hamster}~\cite{hamster} vehicle in Fig.~\ref{fig:hamster}, a $25\times20$~cm mobile robot \change{with electric steering and electric motor speed control}.
The robot is equipped with sensors 
such as a rotating $360$~deg Lidar, an inertial measurement unit, GPS receiver, HD camera, and motor encoders. 
It has Ackermann steering and its kinematic behavior \change{emulates} that of a regular vehicle. 
To evaluate the performance of the automated driving system, we use an Optitrack motion-capture system~\cite{optitrack}, see Fig.~\ref{fig:optitrack}, to obtain position and orientation measurements for each of the Hamster vehicles. 
Depending on the environment and quality of the calibration, the Optitrack system can track the position for each of the Hamster vehicles within $1$~cm and with an orientation error of less than $3$~deg.

Our experimental setup consists of three vehicles driving on a two-lane track shaped as a figure eight, resulting in a traffic intersection as shown in Fig.~\ref{fig:track}. Two Hamsters are designated as \emph{obstacles}, executing a standard PID controller that tracks the center line of the current lane. A traffic intersection coordinator
forces each of the obstacles to stop in front of the intersection for at least three seconds before continuing the execution of the PID lane keeping controller when the intersection is free. The third Hamster is the ego vehicle that is controlled by the multi-layer control architecture shown in Figure~\ref{fig:hamster_architecture}, i.e., the proposed MIP-DM method in combination with an NMPC for reference tracking as described in the next section.
Each of the components in Fig.~\ref{fig:hamster_architecture} is executed in a separate ROS node on a single dedicated desktop computer\footnote{The desktop for vehicle experiments is equipped with an Intel i7-6900K CPU @ 3.20GHz~$\times 8$ processor, 64 GB RAM, and Ubuntu 16.04~LTS.}. 

\begin{figure}[t]
	\centerline{\hbox{
			\includegraphics[width=0.48\textwidth,trim={5mm 0cm 0cm 0cm},clip]{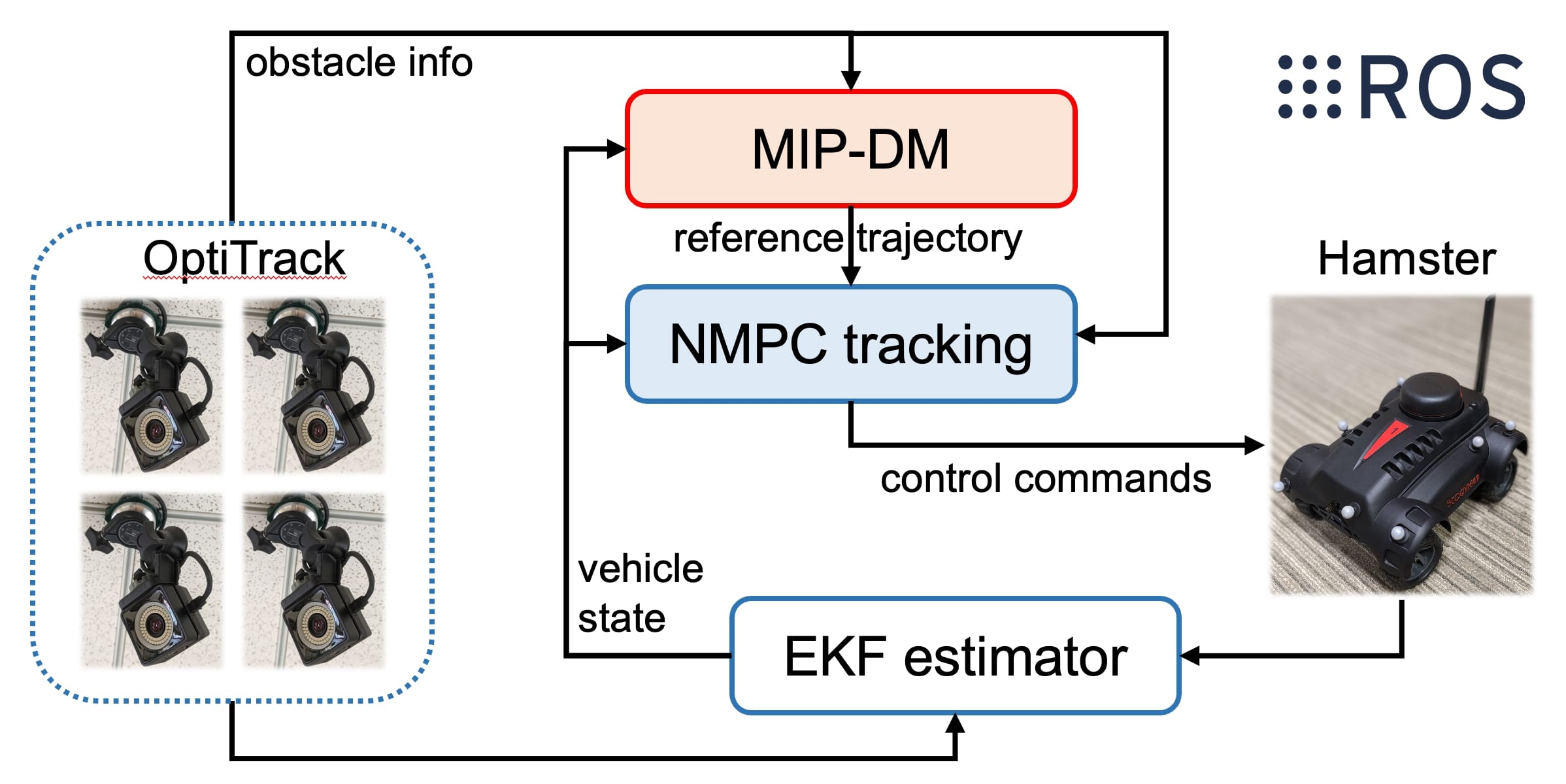}}}		
	\caption{Multi-layer control architecture with MIP-DM, NMPC controller, and EKF state estimator using measurements from the Optitrack system and on-board sensors of the Hamster.}
	\label{fig:hamster_architecture}
\end{figure}

\subsection{Integration of MIP-DM and NMPC Tracking Controller}

We briefly introduce the NMPC that executes the motion plan of the MIP-DM, see Fig.~\ref{fig:hamster_architecture}. Based on the vehicle model in~\eqref{eq:kinematics}, the MIP-DM reference trajectory in curvilinear coordinates is $\begin{bmatrix} \ps(i), \pn(i), \vs(i) \end{bmatrix}^\top$ for $i \in \Z_{0}^{N}$,
which \change{is} transformed to an absolute coordinate frame $(\pX, \pY)$ as in Fig.~\ref{fig:curvlinear}. Given an approximation of the heading angle $\psi(i) \approx \arctan\left(\frac{\pY(i+1) - \pY(i)}{\pX(i+1) - \pX(i)}\right)$, we obtain a reference trajectory $\begin{bmatrix} \pX(i), \pY(i), \psi(i), v(i) \end{bmatrix}^\top$ for $i \in \Z_{0}^{N}$. Similar to~\cite{Quirynen2020,PRESAS}, we use a $3^\text{rd}$ order polynomial approximation, resulting in $\yout^\mathrm{ref}(\tau) = \begin{bmatrix} \pX^\mathrm{ref}(\tau), \pY^\mathrm{ref}(\tau), \psi^\mathrm{ref}(\tau), v^\mathrm{ref}(\tau) \end{bmatrix}^\top$ for $0 \le \tau \le T^{\mathrm{mpc}}$, where $T^{\mathrm{mpc}}$ is the NMPC horizon length.

\change{For the NMPC prediction model,} we use the nonlinear kinematic model~\eqref{eq:nonl_kinematics} with additional actuation dynamics as in~\cite{Quirynen2020}, resulting in the continuous time dynamics
\begin{subequations}
	\begin{alignat}{5}
		\dot{p}_{\mathrm{X}} &= v\, \text{cos}(\psi+\beta), &\quad \dot{p}_{\mathrm{Y}} &= v\, \text{sin}(\psi+\beta), \\
		\dot{\psi} &= v\, \frac{\text{cos}(\beta)}{L} \text{tan}(\delta_{\mathrm{f}}), &\quad \dot{\delta}_{\mathrm{f}} &= \frac{1}{t_{\mathrm{d}}}(\delta + \change{\delta_{\mathrm{o}}} - \delta_{\mathrm{f}}) \\
	    \dot{v} &= u_1, &\quad  \dot{\delta} &= u_2,
	\end{alignat} \label{eq:ct_kinematics}%
\end{subequations}
where $p_{\mathrm{X}},p_{\mathrm{Y}}$ is the longitudinal and lateral position in the world frame, $\psi$ is the heading angle and $\dot \psi$ the heading rate, $v$ is the longitudinal velocity, $\delta$ and $\delta_{\mathrm{f}}$ are the \change{commanded} and actual front wheel steering angle, respectively, and $L, \beta$ are defined as in~\eqref{eq:nonl_kinematics}. 
First order front steering dynamics are included in~\eqref{eq:ct_kinematics} for the steering actuation response. In addition, we estimate the offset value $\change{\delta_{\mathrm{o}}}$ for the steering angle online using an extended Kalman filter~(EKF), \change{which also compensates for unmodeled disturbances,} see Fig.~\ref{fig:hamster_architecture}. The inputs $u_1,u_2$ are the acceleration and steering rate, respectively. 

\change{At each control time step $t$, the NMPC solves}
\begin{subequations} \label{eq:NMPC}%
	\begin{alignat}{5}
		\hspace{-3mm}\underset{X,\,U}{\text{min}} \quad &\frac{1}{2}  \sum_{i=0}^{N^{\mathrm{mpc}}} \Vert \yout(k) - \yout^{\mathrm{ref}}(t_k) \Vert_{Q}^2 + \Vert e_\mathrm{Y}(k) \Vert_{W}^2 && \\
		& \hspace{6mm}+ \Vert u(k) \Vert_{R}^2 + r_{\nu}\, \nu(k) && \label{NMPC:obj}\\
		\hspace{-3mm}\text{s.t.} \quad\; 
		& x(0) = \change{\hat{x}_t}, 	\label{NMPC:initial}\\
		& x(k+1) = f_k\left(x(k), u(k) \right), && \hspace{-3em} \forall k \in \Z_{0}^{N^{\mathrm{mpc}}-1}, \label{NMPC:dyn}\\
		& \ubar{c}_k \le c_k\left(x(k), u(k) \right) \le \bar{c}_k, && \hspace{-3em} \forall k \in \Z_{0}^{N^{\mathrm{mpc}}}, \label{NMPC:ineq}
	\end{alignat}
\end{subequations}
\change{where the} $N^{\mathrm{mpc}}$ control intervals are defined by an equidistant grid of time points $t_k = k \frac{T^{\mathrm{mpc}}}{N^{\mathrm{mpc}}}$ for $k \in \Z_{0}^{N^{\mathrm{mpc}}}$ over the NMPC horizon, $\change{\hat{x}_t}$ is the current state estimate from the EKF \change{at time $t$}, and the constraints in~\eqref{NMPC:dyn} are a discretization of the continuous time dynamics in~\eqref{eq:ct_kinematics} using a $4^\text{th}$~order Runge-Kutta method. The NMPC tracking objective is formulated as a weighted least squares cost of the error between the output $\yout(k)$ and the reference trajectory $\yout^\mathrm{ref}(\tau)$, the path error $e_\mathrm{Y}(k) = \text{cos}(\psi^\mathrm{ref}(t_k))\left(p_{Y}(k)-p_{Y}^\mathrm{ref}(t_k)\right) - \text{sin}(\psi^\mathrm{ref}(t_k))\left(p_{X}(k)-p_{X}^\mathrm{ref}(t_k)\right)$, the squared inputs and an $L_1$ penalty on the slack variables $\nu(k)$.
We introduce a nonnegative slack variable $\nu(k) \ge 0$ for implementing the $L_1$ penalty, and the weight $r_{\nu} \gg 0$ is chosen sufficiently large to ensure that $\nu(k)=0$ when a feasible solution exists~\cite{Fletcher1987}.

Constraints~\eqref{NMPC:ineq} include hard bounds on the control inputs and soft constraints for limiting the distance to the reference trajectory, the velocity and the steering angle
\begin{subequations}
	\begin{alignat}{5}
		-\bar{e}_Y &\le \; e_{Y} + s, \quad
		& -\bar{\delta}_f &\le \; \delta_f + s, \quad
		& -\bar{v} &\le \; v + s, \\
		e_{Y} &\le \bar{e}_Y + s, \quad
		& \delta_f &\le \bar{\delta}_f + s,\quad
		& v &\le \bar{v} + s,\\
		-\bar{\dot{\delta}} &\le \; \dot{\delta} \le 
		\bar{\dot{\delta}}, \quad
		&-\bar{\dot{v}} &\le \; \dot{v} \le 
		\bar{\dot{v}}.
	\end{alignat}\label{eq:softConstr}%
\end{subequations}
\change{In NMPC, obstacle avoidance is enforced by ellipsoidal constraints that approximate the rectangular collision region for each obstacle in the MIP-DM, see Fig.~\ref{fig:obstacle_avoidance},}
\begin{equation}
	1 \le \left(\frac{\delta_{x,j}(k)}{a_{x,j}}\right)^2 + \left(\frac{\delta_{y,j}(k)}{a_{y,j}}\right)^2, \label{eq:ellipsoidal}
\end{equation}
where $\begin{bmatrix}\delta_{x,j} \\ \delta_{y,j}\end{bmatrix} = R(o_{\psi,j})^\top\begin{bmatrix}p_{\mathrm{X}} - o_{\mathrm{X},j} \\ p_{\mathrm{Y}} - o_{\mathrm{Y},j}\end{bmatrix}$ is the rotated distance, $(o_{\mathrm{X},j},o_{\mathrm{Y},j},o_{\psi,j})$ is the obstacle's pose, and $(a_{x,j},a_{y,j})$ are the lengths of the principal semi-axes of the ellipsoid that ensure a safety margin around each obstacle. 

The nonlinear OCP~\eqref{eq:NMPC} includes $\nx=6$ states, $\nU=3$ control inputs and $N^{\mathrm{mpc}}=80$ control intervals with a sampling period of $\Ts^{\mathrm{mpc}} = 25$~ms over a $T^{\mathrm{mpc}}=2$~s horizon length. The NMPC controller is implemented with a sampling frequency of $40$~Hz, using the real-time iteration~(RTI) algorithm~\cite{Gros2016} in the \software{ACADO} code generation tool~\cite{Quirynen2014a} and the \software{PRESAS} QP solver~\cite{PRESAS}. 
\change{The sampling period of MIP-DM is reduced with respect to that of Section~\ref{sec:simulation} due to the scaling of the vehicles. MIP-DM executes with a sampling period of $\Ts^{\mathrm{mip}} = 0.3$~s and horizon length $N^{\mathrm{mip}}=15$.} 

\subsection{Experimental Results using Small-scale Vehicles}

Based on the MIP-DM in Section~\ref{sec:MIQP_form}, the capabilities of the ego vehicle include lane selection, lane change execution, swaying maneuvers, queuing behavior and stopping / crossing at the traffic intersection. Based on the zone constraints in the MIP~(see Section~\ref{sec:zones}), we implement a traffic rule that the ego vehicle is only allowed to make lane changes in the bottom right loop of the figure eight track~(see Fig.~\ref{fig:track}).

Figure~\ref{fig:visualization_figs} shows four snapshots of the experiment. The left side of each subfigure shows the location of the ego~(blue) and two obstacles~(red) on the eight shaped track, the safety ellipsoid around each obstacle~(dashed red line), the NMPC predicted trajectory~(blue plus markers) and the MIP-DM reference trajectory~(magenta circles). The bottom right side of each subfigure in Fig.~\ref{fig:visualization_figs} illustrates the proposed MIP-DM, i.e., it shows the two-lane road in curvilinear coordinates, the location of the ego~(blue), two obstacles~(red), the traffic intersection~(purple), and the MIP solution trajectory~(blue solid circles) over a $T^{\mathrm{mip}}=4.5$~s horizon length. For each obstacle, the dark red~(or dark purple) \change{region} represents the physical shape of the obstacle, while the larger shaded area corresponds to the avoidance constraints in the MIP-DM. A sequence of larger shaded areas is shown for each obstacle based on a prediction of the obstacle behavior over the MIP-DM horizon. The top right side of each subfigure in Fig.~\ref{fig:visualization_figs} shows the steering angle and velocity command in the NMPC control input trajectory over a $T^{\mathrm{mpc}}=2$~s horizon.

Fig.~\ref{fig:visualization_fig1} shows the trajectories for MIP-DM and NMPC at $26$~s in the experiment, demonstrating the ego vehicle stopping at the traffic intersection. After the obstacle~(Hamster~$3$) finishes crossing the intersection, the ego continues by crossing the intersection at $33$~s in the experiment. Fig.~\ref{fig:visualization_fig2} shows the trajectories at $61$~s, demonstrating the ego changing lane and overtaking a slower obstacle to achieve the desired velocity of $0.4$~m/s. 
Fig.~\ref{fig:visualization_fig3} shows the trajectories at $69$~s, demonstrating the ego changing lane back to the preferred lane after overtaking the slower obstacle. Finally, Fig.~\ref{fig:visualization_fig4} shows the trajectories at $183$~s in the experiment, demonstrating the ego queuing behind a slower obstacle because overtaking is not allowed in the top left loop of the figure eight track.

\begin{figure*}
\centering
\hspace{-1cm}
\begin{minipage}{0.45\linewidth}%
	\begin{subfigure}[b]{\textwidth}
		\centering
		\includegraphics[width=1.15\textwidth,trim={2cm 6cm 0 2cm},clip]{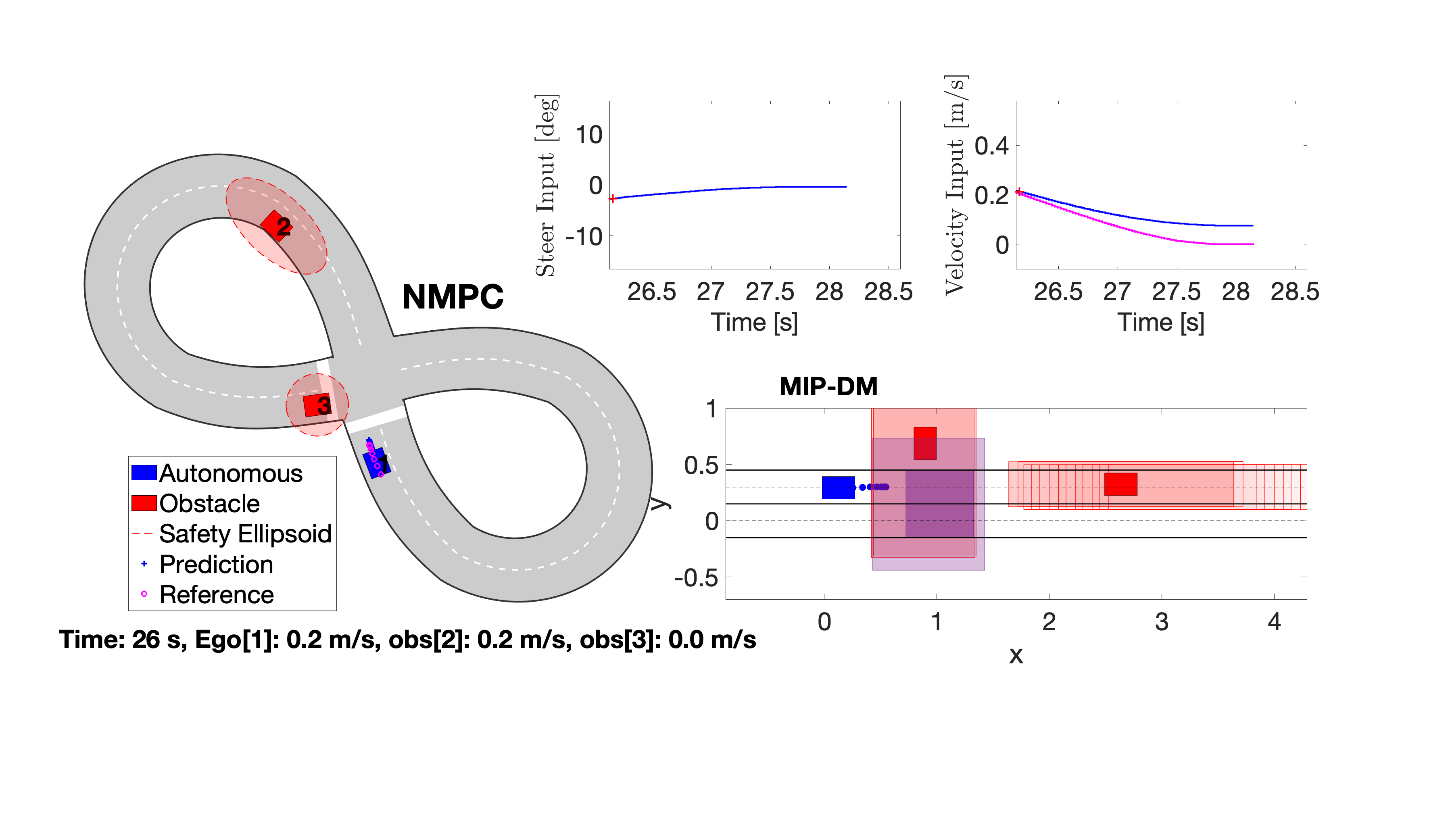}
		\captionsetup{font=normalsize,labelfont={sf}}
		\caption[]%
		{\normalsize Trajectories for MIP-DM and NMPC at $26$~s of experiment: ego vehicle stopping at traffic intersection.}
		\label{fig:visualization_fig1}
	\end{subfigure}
	\vskip\baselineskip\vspace{2mm}
	\begin{subfigure}[b]{\textwidth}
		\centering
		\includegraphics[width=1.15\textwidth,trim={2cm 6cm 0 2cm},clip]{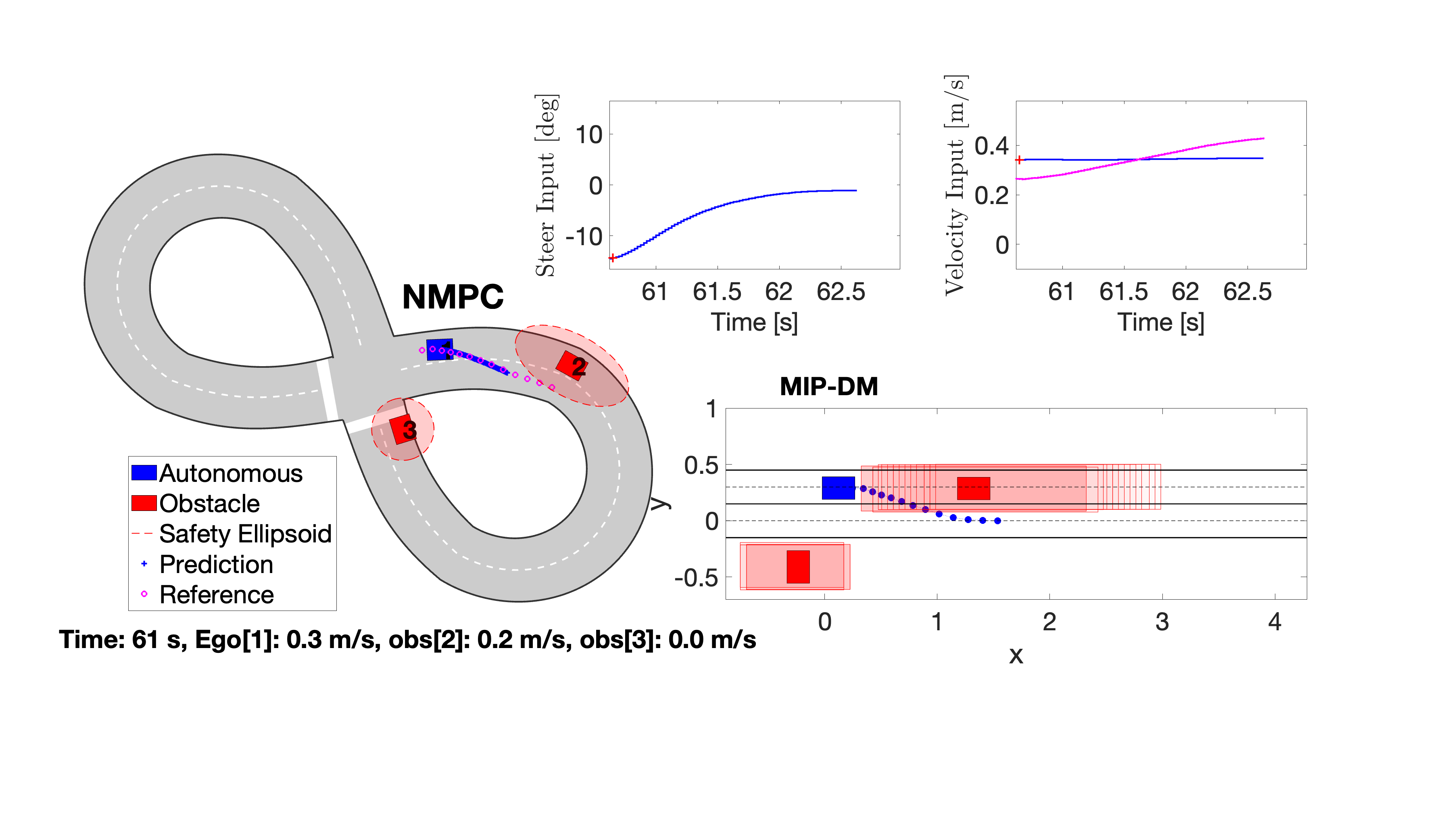}
		\captionsetup{font=normalsize,labelfont={sf}}
		\caption[]%
		{\normalsize Trajectories for MIP-DM and NMPC at $61$~s of experiment: ego vehicle overtaking slower obstacle to achieve desired velocity.}
		\label{fig:visualization_fig2}
	\end{subfigure}
\end{minipage}
\begin{minipage}{0.05\linewidth}%
	\hspace{2mm}
\end{minipage}
\begin{minipage}{0.45\linewidth}%
\begin{subfigure}[b]{\textwidth}
	\centering
	\includegraphics[width=1.15\textwidth,trim={2cm 6cm 0 2cm},clip]{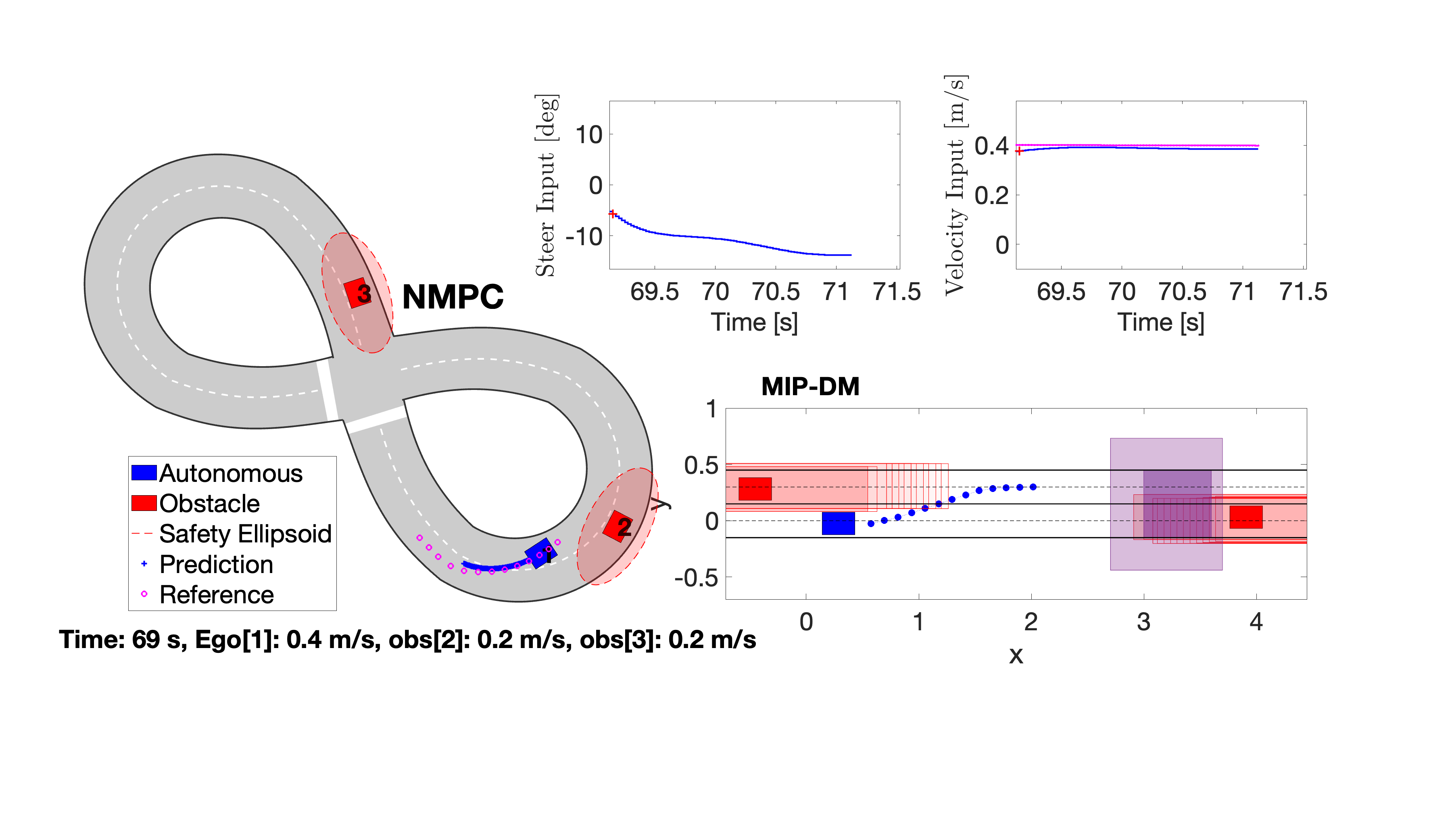}
	\captionsetup{font=normalsize,labelfont={sf}}
	\caption[]%
	{\normalsize Trajectories for MIP-DM and NMPC at $69$~s of experiment: ego vehicle returning to preferred lane after overtaking slower obstacle.}
	\label{fig:visualization_fig3}
\end{subfigure}
\vskip\baselineskip
\begin{subfigure}[b]{\textwidth}
	\centering
	\includegraphics[width=1.15\textwidth,trim={2cm 6cm 0 3cm},clip]{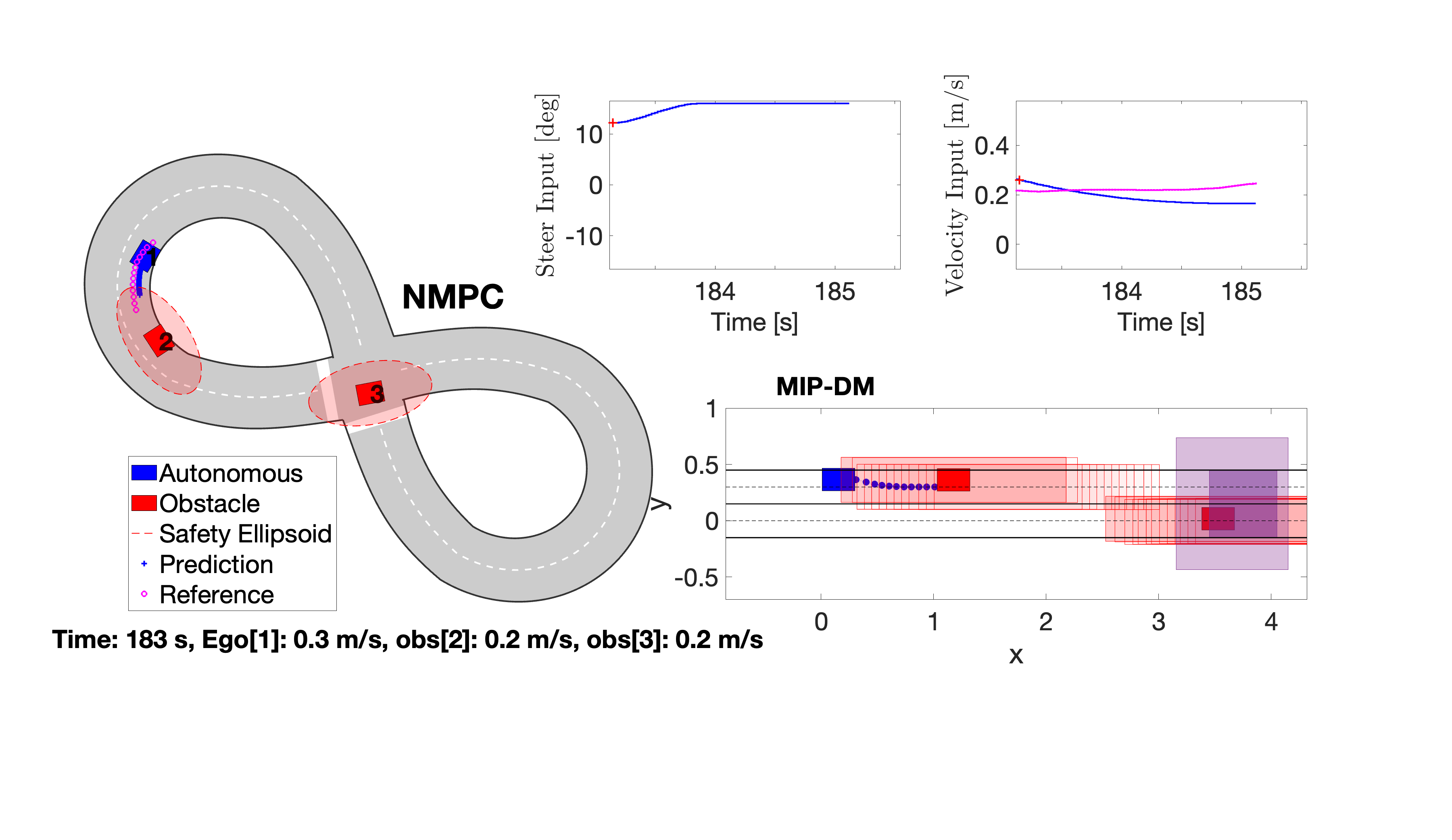}
	\captionsetup{font=normalsize,labelfont={sf}}
	\caption[]%
	{\normalsize Trajectories for MIP-DM and NMPC at $183$~s of experiment: ego vehicle slowing down behind slower obstacle because overtaking is not allowed.}
	\label{fig:visualization_fig4}
\end{subfigure}
\end{minipage}
\hspace{-1cm}
\caption[ ]{Illustration of predicted trajectories of MIP-DM~($\Ts^{\mathrm{mip}}=0.3$~s), and NMPC~($\Ts^{\mathrm{mpc}}=0.025$~s) tracking the MIP-DM reference, at certain steps of small-scale vehicle experiments. \change{
The left side of each subfigure shows the eight shaped track, the ego~(blue) and two obstacles~(red), safety ellipsoid around each obstacle~(dashed red line), NMPC predicted trajectory~(blue plus markers) and MIP-DM reference~(magenta circles). The bottom right side of each subfigure shows the ego~(blue), two obstacles~(red), traffic intersection~(purple), and MIP-DM solution~(blue solid circles) in curvilinear coordinates, and the top right side shows the NMPC control input trajectory.}
A video is available at: \url{https://youtu.be/FyaGRZvuqmA}.}
\label{fig:visualization_figs}
\end{figure*}

Figure~\ref{fig:hamster_trace} shows the trace of ego positions~(in blue) during the $200$~s experiment, and each of the locations where the ego vehicle came to a full stop are highlighted by red dots. The ego vehicle consistently stops at a desired safety distance from the intersection before crossing. The one red dot away from the intersection is due to the queuing behavior in Fig.~\ref{fig:visualization_fig4}, where the ego stops behind an obstacle at the intersection. In addition, Fig.~\ref{fig:hamster_trace} confirms that the ego vehicle only makes lane changes in the bottom right loop of the track, demonstrating the zone-dependent traffic rules in Section~\ref{sec:zones}. Finally, Figure~\ref{fig:CPU_times} shows the CPU times for the \software{BB-ASIPM} solver to implement the MIP-DM during the $200$~s experiment. The computation times are always below $120$~ms and therefore real-time feasible, due to the sampling period of $\Ts^{\mathrm{mip}} = 300$~ms. 

\begin{figure}[t]
	\centerline{\hbox{
			\includegraphics[width=0.38\textwidth,trim={13cm 5cm 11cm 2cm},clip]{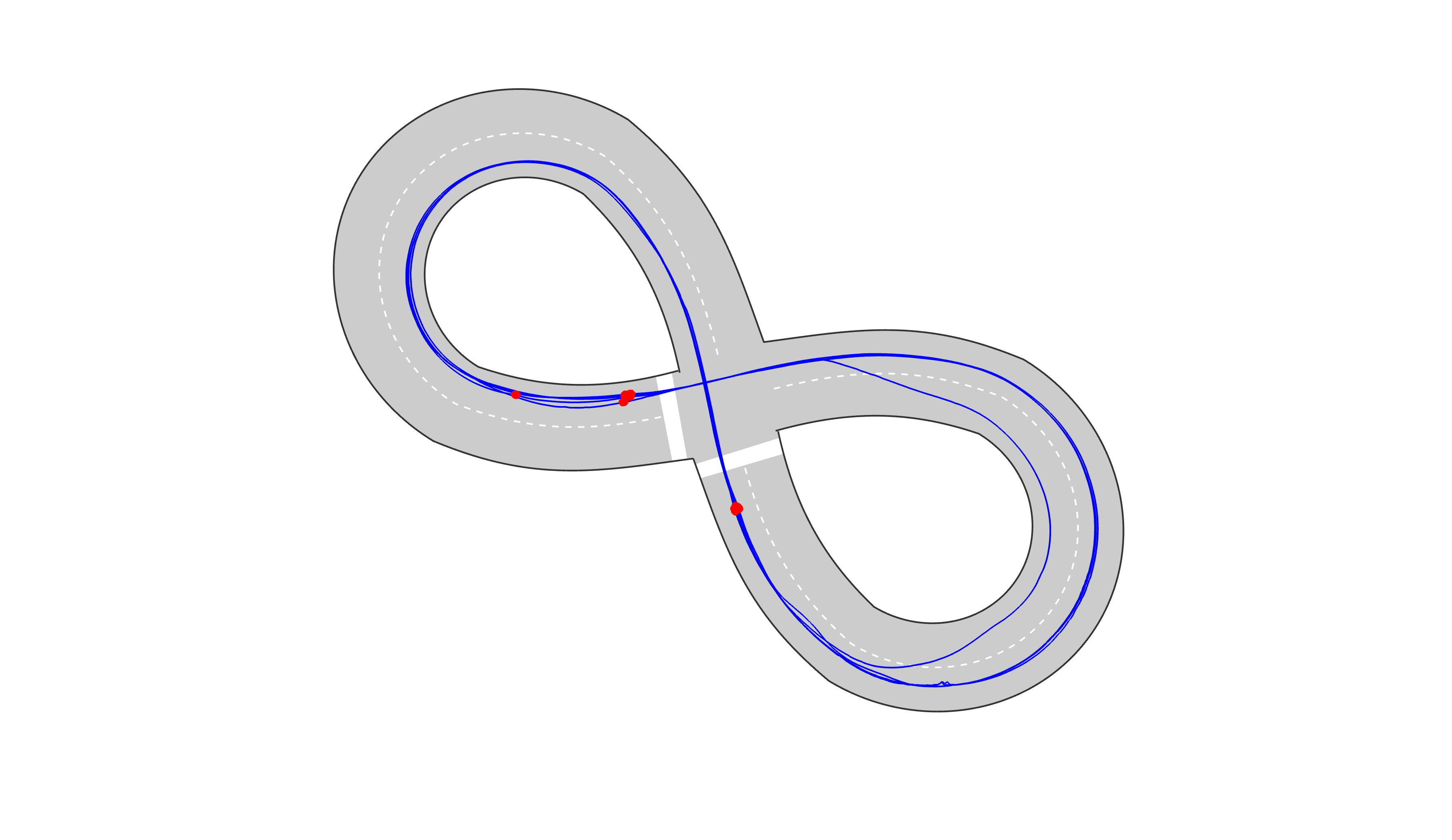}}}		
	\caption{Trace of ego vehicle positions during experiments in Fig.~\ref{fig:visualization_figs}: red dots indicate positions at which the ego stopped, either at the intersection or queuing behind an obstacle.}
	\label{fig:hamster_trace}
\end{figure}

\begin{figure}[t]
	\centerline{\hbox{
			\includegraphics[width=0.5\textwidth,trim={2cm 0cm 1cm 0cm},clip]{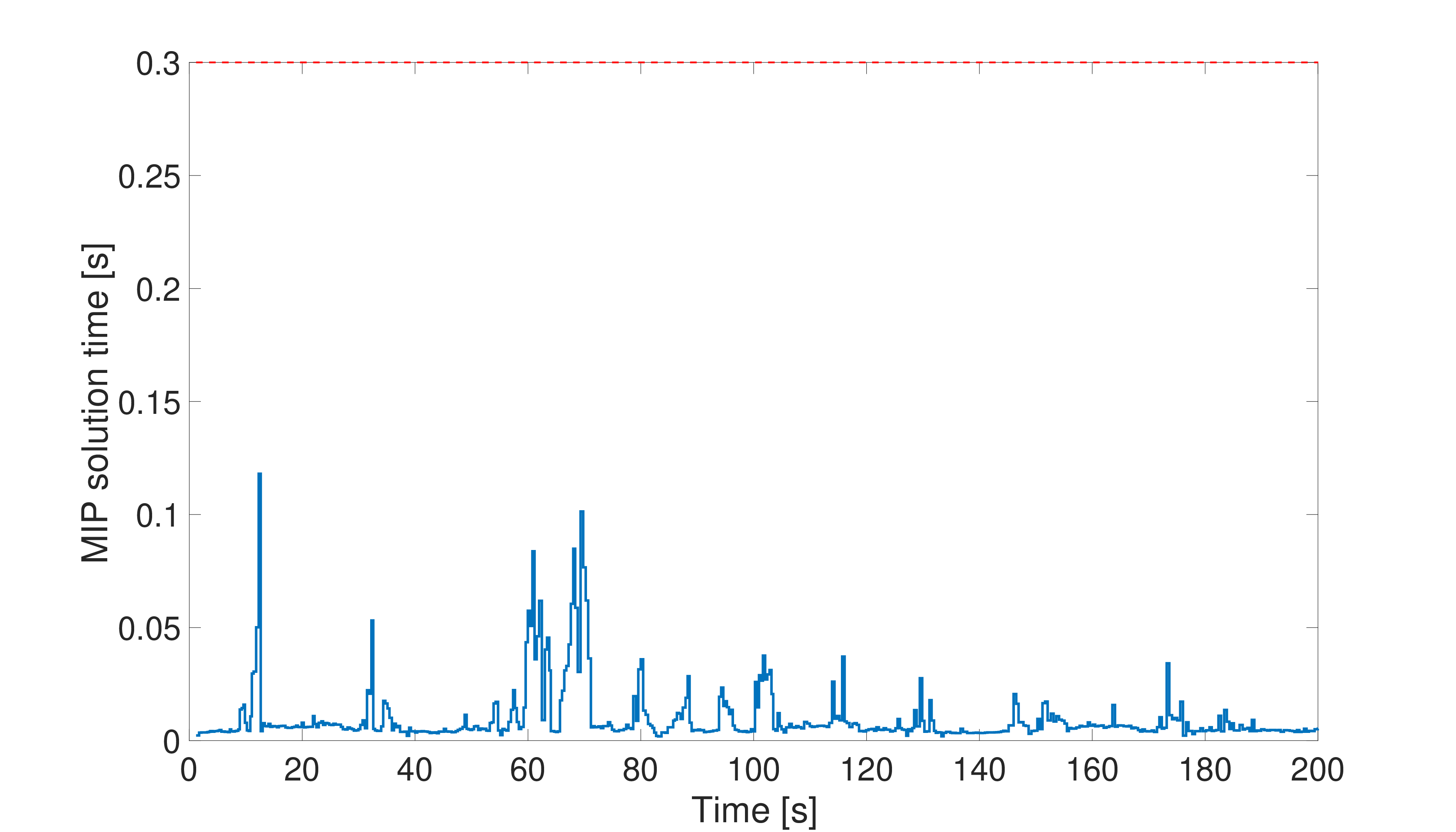}}}		
	\caption{CPU time of \software{BB-ASIPM} solver in MIP-DM (sampling period $\Ts^{\mathrm{mip}} = 0.3$~s) during the experiments in Fig.~\ref{fig:visualization_figs}.}
	\label{fig:CPU_times}
\end{figure}

\section{Conclusions and Outlook} \label{sec:concl}
We designed a mixed-integer programming-based decision making for automated driving. The mixed-integer quadratic programming formulation uses a linear vehicle model in a road-aligned coordinate frame, it includes lane selection and lane change timing constraints, polyhedral collision avoidance and intersection crossing constraints, and zone-dependent traffic rule changes. \change{We leveraged the recently developed} embedded \software{BB-ASIPM} solver, using a branch-and-bound method with reliability branching and warm starting, block-sparse tailored presolve techniques, early termination and infeasibility detection within an active-set interior point method. The performance of the MIP-DM method was demonstrated by simulations in various scenarios including merging points and traffic intersections, and real-time feasibility was demonstrated by hardware-in-the-loop simulations 
on dSPACE Scalexio and MicroAutoBox-III rapid prototyping units. Finally, we presented results from experiments 
on a setup with small-scale vehicles, integrating the MIP-DM with a nonlinear model predictive control for reference tracking.

Future \change{works will focus on} using more advanced behavior prediction models for other vehicles and explicit handling of uncertainty in the modeling and perception of the environment, \change{as well as deployment on full scale vehicles}.

\bibliographystyle{IEEEtran}
\bibliography{references}

\begin{IEEEbiography}[{\includegraphics[width=1in,height=1.25in,clip,keepaspectratio]{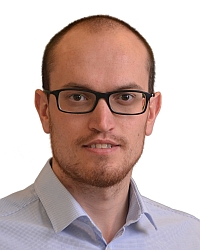}}]{Rien Quirynen} received the Bachelor’s degree in computer science and electrical engineering and the Master’s degree in mathematical engineering from KU Leuven, Belgium. He received a four-year Ph.D. Scholarship from the Research Foundation–Flanders~(FWO) in 2012-2016, and the joint Ph.D. degree from KU Leuven, Belgium and the University of Freiburg, Germany.
Since 2017, he joined Mitsubishi Electric Research Laboratories in Cambridge, MA, USA, where he is currently a senior principal research scientist. His research focuses on numerical optimization algorithms for decision making, motion planning and predictive control of autonomous systems. He has authored/coauthored more than 75 peer-reviewed papers in journals and conference proceedings and 25 patents. Dr. Quirynen serves as an Associate Editor for the Wiley journal of Optimal Control Applications and Methods and for the IEEE CCTA Editorial Board. 
\end{IEEEbiography}

\begin{IEEEbiography}[{\includegraphics[width=1in,height=1.25in,clip,keepaspectratio,trim={18cm 18cm 18cm 18cm},clip]{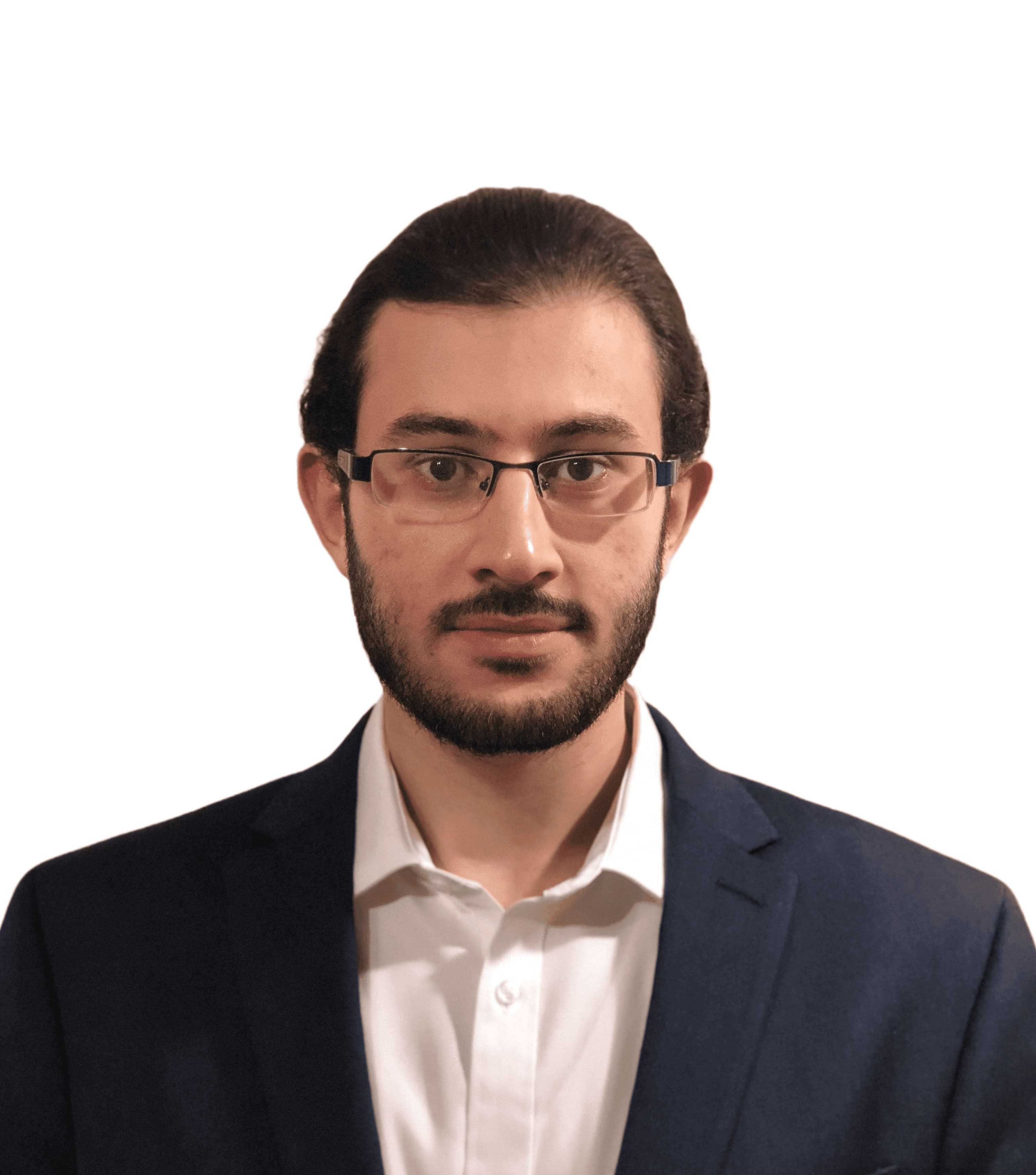}}]{Sleiman Safaoui} received the B.S. and M.S. degrees in electrical engineering from the University of Texas at Dallas, Richardson, TX, USA, in 2019 and 2023, respectively. He is currently pursing his Ph.D. degree in electrical engineering at the University of Texas at Dallas, Richardson, TX, USA as a Research Assistant with the Control, Optimization, and Networks Lab (CONLab). From Aug 2021 to March 2022, he was an intern at Mitsubishi Electric Research Laboratories (MERL) where he worked on ground and aerial vehicle autonomy research projects. His current research interests include risk-based motion planning and control for robotic systems under uncertainty, autonomus vehicles, and multirobot systems.
\end{IEEEbiography}

\begin{IEEEbiography}[{\includegraphics[width=1in,height=1.25in,clip,keepaspectratio]{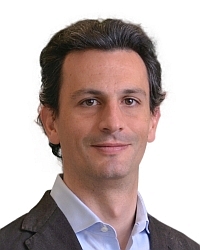}}]{Stefano Di Cairano} %
	received the Master’s (Laurea) and the Ph.D. degrees in information engineering in 2004 and 2008, respectively, from the University of Siena, Italy. 
	During 2008-2011, he was with Powertrain Control R\&A, Ford Research and Advanced Engineering, Dearborn, MI, USA. Since 2011, he is with Mitsubishi Electric Research Laboratories, Cambridge, MA, USA, where he is currently a Deputy Director, and a Distinguished Research Scientist. His research focuses on optimization-based control and decision-making strategies for complex mechatronic systems, in automotive, factory automation, transportation systems, and aerospace. His research interests include model predictive control, constrained control, path planning, hybrid systems, optimization, and particle filtering. He has authored/coauthored more than 200 peer-reviewed papers in journals and conference proceedings and 70 patents.
	Dr. Di Cairano was the Chair of the IEEE CSS Technical Committee on Automotive Controls and of the IEEE CSS Standing Committee on Standards. He is the inaugural Chair of the IEEE CCTA Editorial Board and was an Associate Editor of the IEEE \sc{Transactions on Control Systems Technology}. 
\end{IEEEbiography}

\end{document}